\documentclass[a4paper,11pt,reqno,openany]{amsart}
\usepackage[dvipsname,usenames]{xcolor}
\usepackage[utf8]{inputenc}
\usepackage[T1]{fontenc}
\usepackage{lmodern}
\usepackage[english]{babel}
\usepackage{microtype}
\usepackage{pdfsync}
\usepackage{changepage}
\usepackage{bbm}
\usepackage{pgf,tikz}

\newcommand{\nocontentsline}[3]{}
\newcommand{\tocless}[2]{\bgroup\let\addcontentsline=\nocontentsline#1{#2}\egroup}

\usepackage{tikz-cd}

\usepackage{enumerate}
\addto\extrasenglish{%
}

\usepackage{amsmath,amssymb,amsfonts,amsthm}
\usepackage{mathtools}
\usepackage{accents}
\usepackage{mathrsfs}
\usepackage{aliascnt}
\usepackage{braket}
\usepackage{bm}

\usepackage{esint}

\usepackage[a4paper,margin=2.5cm]{geometry}
\usepackage[citecolor=blue,colorlinks]{hyperref}

\usepackage{enumerate}
\usepackage{xcolor}

\makeatletter
\g@addto@macro\@floatboxreset\centering
\makeatother

\makeatletter
\def\newaliasedtheorem#1[#2]#3{
  \newaliascnt{#1@alt}{#2}
  \newtheorem{#1}[#1@alt]{#3}
  \expandafter\newcommand\csname #1@altname\endcsname{#3}
}
\makeatother

\numberwithin{equation}{section}

\DeclareMathOperator*{\argmin}{arg\,min}

\newtheoremstyle{slanted}{\topsep}{\topsep}{\slshape}{}{\bfseries}{.}{.5em}{}

\theoremstyle{plain}
\newtheorem{theorem}{Theorem}[section]
\newaliasedtheorem{proposition}[theorem]{Proposition}
\newaliasedtheorem{lemma}[theorem]{Lemma}
\newaliasedtheorem{corollary}[theorem]{Corollary}
\newaliasedtheorem{counterexample}[theorem]{Counterexample}

\theoremstyle{definition}
\newaliasedtheorem{definition}[theorem]{Definition}
\newaliasedtheorem{question}[theorem]{Question}
\newaliasedtheorem{example}[theorem]{Example}

\theoremstyle{remark}
\newaliasedtheorem{remark}[theorem]{Remark}

\newcommand{\setN}{\mathbb{N}}

\newcommand{\setR}{\mathbb{R}}

\newcommand{\N}{\mathbb{N}}

\newcommand{\R}{\mathbb{R}}
\newcommand{\norm}[1]{\left\lVert#1\right\rVert}

\newcommand{\weakto}{\rightharpoonup}

\DeclareMathOperator{\Lip}{Lip}

\newcommand{\Leb}{\mathscr{L}}

\newcommand{\Prob}{\mathscr{P}}

\newcommand{\tL}{\text{L}}

\newcommand{\res}{\mathop{\hbox{\vrule height 7pt width .5pt depth 0pt
\vrule height .5pt width 6pt depth 0pt}}\nolimits}

\newcommand{\suchthat}{\ensuremath{\ : \ }} 
\newcommand{\de}{\ensuremath{\, \mathrm d}} 
\DeclareMathOperator{\supp}{supp} 

\newcommand{\E}{\mathcal{E}}
\newcommand{\cE}{\mathcal{E}}

\newcommand{\cP}{\Prob}
\newcommand{\cK}{\mathcal{K}}
\newcommand{\cC}{\mathcal{C}}

\newcommand{\cPC}{\mathcal{PC}}

\newcommand{\bpW}{\mathbb{W}_{p,m}}
\newcommand{\btwoW}{\mathbb{W}_{2,m}}

\newcommand{\bNW}{\mathbb{W}_{p,m}^N}
\newcommand{\bNtwoW}{\mathbb{W}_{2,m}^N}
\newcommand{\bNhW}{\widehat{\mathbb{W}}_{p,m}^N}
\newcommand{\bNhtwoW}{\widehat{\mathbb{W}}_{2,m}^N}
\newcommand{\bW}{\mathbb{W}}
\newcommand{\Sph}{\mathcal{S}_{\varphi}}
\newcommand{\cR}{\mathcal{R}}
\newcommand{\cM}{\mathcal{M}}

\newcommand{\cD}{\mathcal{D}}

\newcommand{\diam}{\text{diam}}

\def\BBB{\color{blue}}	
\def\BLK{\color{black}\normalsize}

\newcommand{\eps}{\varepsilon}
\newcommand{\vphi}{\varphi}
\newcommand{\inte}{\mathrm{int}}
\newcommand{\dom}{\mathrm{dom}}
\newcommand{\If}{\mathrm{if} \;}

\newcommand{\cJ}{\mathcal J}

\newcommand{\bE}{\mathbb E}

\newcommand{\bCE}{\mathsf{CE}}
\newcommand{\CE}{\mathsf{CE}}

\def\BS{\boldsymbol}
\def\bfmu{{\BS\mu}}
\def\bfnu{{\BS\nu}}            

\newcommand{\TN}{\text{T}\cK_N}
\newcommand{\xTN}[1]{\text{T}_{#1}\cK_N}

\newcommand{\tand}{\quad \text{and} \quad }
\newcommand{\seq}[1]{\left( #1 \right)}

\title[Optimal transport with non
linear mobilities]{Optimal transport with nonlinear mobilities: a deterministic particle approximation result}
\author{S. Di Marino}
\address{Università di Genova, DIMA, MaLGa, Via Dodecaneso 35, 16146 Genova (GE), Italy}
\email{simone.dimarino@unige.it}

\author{L. Portinale}
\address{Haussdorff Center for Mathematics, Institute for Applied Mathematics, Endenicher Allee 62, D-53115 Bonn, Germany}
\email{portinale@iam.uni-bonn.de}

\author{E. Radici}
\address{Università degli Studi dell'Aquila, DISIM, Via Vetoio 1, 67100 (Coppito) L'Aquila (AQ), Italy}
\email{emanuela.radici@univaq.it}

\date{\today}
\begin{document}
\maketitle
\begin{abstract}
	We study the discretisation of generalised Wasserstein distances with nonlinear mobilities
	on the real line via suitable discrete metrics on the cone of N ordered particles, a setting which naturally appears in the framework of deterministic particle approximation of partial differential equations. In particular, we
	provide a $\Gamma$-convergence result for the associated discrete metrics as $N \to \infty$ to the continuous one	and discuss applications to the approximation of one-dimensional conservation laws (of gradient
	flow type) via the so-called generalised minimising movements, proving a convergence result of the schemes at any given discrete time step $\tau>0$. This the first work of a series aimed at sheding new lights on the interplay between generalised gradient-flow structures, conservation laws, and Wasserstein distances with nonlinear mobilities.
\end{abstract}
\setcounter{tocdepth}{1}
\tableofcontents

\section{Introduction}

The optimal transport theory has been a very active research field in the last twenty years, revealing suprising connections between the classical Monge--Kantorovich transportation model and various other fields of mathematics, including the theory of partial differential equations (PDEs) \cite{jordan1998variational}, functional and geometric inequalities \cite{otto2000generalization}, and metric measure spaces \cite{lott2007weak,Lott-Villani:2009,Sturm:2006:I}.
These connections are fundamentally based on the ideas of the seminal works of Benamou--Brenier \cite{BeBr00} and Otto \cite{otto2001geometry}: one can associate a (formal) infinite-dimensional Riemannian structure to the space of probability measures using the quadratic transport distance (i.e. the $2$-Wasserstein--Kantorovich--Rubenstein distance), whose associated gradient flow of the relative entropy coincides with the heat equation (in $\setR^d$). This new perspective opened the door to wide range applications, not limited to the continuous setting of Riemannian manifolds, but applicable to discrete and non-commutative frameworks as well \cite{maas2011gradient, mielke2011gradient,carlen2014analog}.

\vspace{2mm}
The main focus of this work is the study of generalised optimal transport distances \cite{dolbeault2012}, \cite{lisiniMarigonda2010} on the real line, its spatial discretisation using deterministic particle methods (in the spirit of  \cite{difrancesco2015}), and the applications of these methods to the study of the associated gradient-flow evolutions. 

The class of evolutions we are interested in might include local and/or non local effects, possibily including prevention of overcrowding. Mathematically, this translates into considering evolutions with \textit{non-linear mobilities}, possibly with bounded support. To these equations, one can associate a corresponding Wasserstein-like distance that takes into account the nonlinearity. 

Our goal is to study suitable discretisations of these problems, both the generalised optimal transport distances and the associated gradient flows, with particular attention to two types of approximations: in space and in time.

In particular, we introduce, in the scalar case, a \textit{space discretisation} in the framework of non-linear mobilities, adopting a \textit{Lagrangian} point of view. The generalised transport distances are approximated using systems of $N$-ordered particles, and a discrete-to-continuum approximation result is provided (Theorem~\ref{thm:Gamma-conv}). Subsequently, we take advantage of this discrete-in-space approximation to show the stability at the level of the corresponding gradient-flow structures, providing a finite-dimensional approximation 
of the associated evolution equation (Theorem~\ref{thm:JKO_q}).

Our contribution aim towards the understanding of the discrete-to-continuum time and space \textsl{commuting diagram} regarding the one dimensional evolutionary PDEs that can be seen as gradient flows with respect to a non-linear mobility Wasserstein like distance, as described in Figure~\ref{fig:1}.

\,

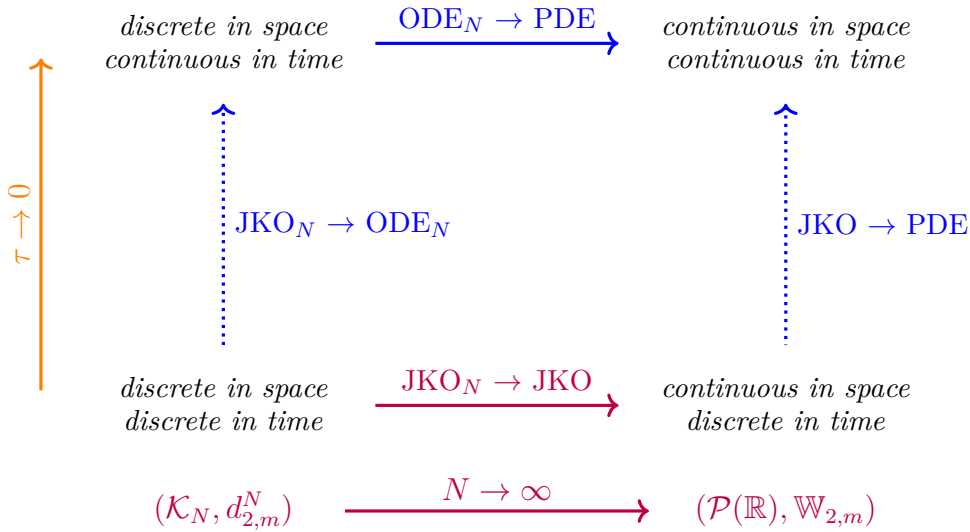
\begin{figure}[h]
	\label{fig:1}
\begin{tikzpicture}[scale=4]
\tikzstyle{every node}=[font=\fontsize{38}{8}]
%
	\draw[black] (0,2) node {\normalsize \textit{discrete in space}};
	\draw[black] (0,1.9) node {\normalsize \textit{continuous in time}};
		\draw[blue,very thick, ->]
		(0.5, 1.95) -- (1.3,1.95) node [above, midway] {\normalsize ODE$_N$ $\to$ PDE};
	\draw[black] (1.85,2) node  {\normalsize \textit{continuous in space}};
	\draw[black] (1.85,1.9) node {\normalsize \textit{continuous in time}};
		\draw[blue,very thick, <-, dotted]
		(1.85, 1.75) -- (1.85,0.95)
		node [right,midway] {\normalsize JKO $\to$ PDE} ;
	\draw[black] (0,0.8) node {\normalsize \textit{discrete in space}};
	\draw[black] (0,0.7) node {\normalsize \textit{discrete in time}};
		\draw[blue,very thick, ->, dotted]
		(0, 0.95) -- (0,1.75)
		node [right,midway] {\normalsize JKO$_N$ $\to$ ODE$_N$} ; 
	\draw[orange,very thick, ->]
		(-0.6, 0.8) -- (-0.6,1.9) node [midway, above, sloped] (TextNode) {\large $\tau \to 0$};  
	\draw[purple] (0,0.4) node {\large $(\mathcal K_N, d_{2,m}^N)$};
		\draw[purple,very thick, ->]
		(0.4, 0.4) -- (1.4,0.4)  node [midway, above, sloped] (TextNode) {\large $N \to \infty$};
	\draw[purple] (1.85,0.4) node {\large $(\mathcal P(\mathbb R), \mathbb W_{2,m})$};
	\draw[purple,very thick, ->]
		(0.5, 0.75) -- (1.3,0.75) 
		node [above, midway] {\normalsize JKO$_N$ $\to$ JKO};
	\draw[black] (1.85,0.8) node {\normalsize \textit{continuous in space}};
		\draw[black] (1.85,0.7) node {\normalsize \textit{discrete in time}};
	\end{tikzpicture}
\caption{The discrete-to-continuum diagram. }
\end{figure}
This work focus on the purple arrows, concerning the convergence of the discrete-in-time schemes and the discrete distances $d_{2,m}^N$ (Section~\ref{sec:setting}) on the cones $\cK_N$ as $N \to \infty$. The top, blue arrow concerns deterministic many particle approximation of PDEs, and it has been already investigated in several works, see Section~\ref{sec:literature} for a more detailed discussion. The natural continuation of this manuscript would consist in the analysis of the two vertical dotted arrows in Figure~\ref{fig:1}. This is currently under investigation, and we postpone it to future works, since it seems to require a different type of analysis. 

\subsection{Optimal transport with non-linear mobilities}

The \textit{generalised Wasserstein distances} were introduced in \cite{dolbeault2012}, \cite{lisiniMarigonda2010}.  Given $\mu_0,\mu_1 \in \Prob(\Omega)$ and a mobility $m:\R_+ \to \R_+$, the generalized Wasserstein psuedo-distance is defined by
\begin{align*}
\btwoW (\mu_0, \mu_1)^2  :=  \inf \left\{ \int_0^1 \int_{\setR^d} |v_t|^2  m(\rho_t) \de \Leb^1 \de t \suchthat \partial_t \mu_t + \nabla \cdot (m(\rho_t) v_t) =0, \ \mu_{t=i} = \mu_i  \right\} ,
\end{align*}
for $i=0,1$, where in the formula above $\rho_t$ denotes the density of $\mu_t$ with respect to the Lebesgue measure $\Leb^1$ on $\R$.

One of the motivation for the introduction of $\btwoW$ is precisely to be able to give a variational structure to PDEs of the type
\begin{align}	\label{eq:gWGF}
\partial_t \rho_t - \nabla \cdot (m(\rho_t) \nabla (\text{D} \cE(\rho_t)) = 0 \, , 
\end{align}
which indeed can be formally seen as gradient flows of $\cE$ with respect to $\btwoW$. The linear case $m(\rho)=\rho$ corresponds to the classical Wasserstein distance $\bW_2$.

Following the ideas of Jordan--Kinderlehrer--Otto, one 
rigorous interpretation of \eqref{eq:gWGF} as a gradient flow 
consists in considering the time-discrete approximation (also known as \textit{JKO-scheme} from the names of the authors, or, more generally, \textit{minimising movements scheme} in the spirit of De Giorgi, see \cite{AmbrosioGigliSavare08,Santambrogio:2017}) of parameter $\tau>0$ given by
\begin{align}	\label{eq:JKOnonlinear}
\mu_\tau^{(n)} \in \argmin_{\mu \in \Prob(\setR^d)} \left\{ \cE(\mu) + \frac1{2\tau}\btwoW(\mu, \mu_\tau^{(n-1)})^2 \right\} \, .
\end{align}
The goal is then to recover a solution \eqref{eq:gWGF} defining a curve $\bm \mu^\tau: [0,1] \to \cP(\setR^d)$ obtained, for every $\tau>0$, interpolating the family $\seq{ \mu_\tau^{(n)}  }_n$ and taking its limit as $\tau \to 0$. This is the content of \cite{jordan1998variational} in the case of linear mobility and the heat equation in $\setR^d$.

\subsection{Main results}

The goal of this work is a particle approximation result for the Wasserstein-like distances with non-linear mobilities $\btwoW$ and its applications to the convergence of associated gradient-flow structures, in the scalar case $d=1$.
A natural generalisation of this result to the case of $p$-homogenous energy functionals for $p \neq 2$ is also considered. For the sake of simplicity, we here restrict the exposition of the main results to the quadratic case.

The starting point of discussion is the work  \cite{difrancesco2015} (together with several companion works, see Section~\ref{sec:literature} for a more detailed literature review), where the authors propose a deterministic particle approximation result to entropic solutions to scalar conservation laws of the form $\partial_t \mu_t + \partial_x (f[\mu_t])=0$. The procedure is based on a discretisation of the Lagrangian formulation of the evolution, working in terms of pseudo-inverses of cumulative distributions and their approximation. For any $\mu \in \Prob(\setR)$, one considers the quantile function 
\begin{align*}
	X_\mu(z) :=
	 \inf \left \{ a \in \setR \suchthat \mu (-\infty, a) > z \right\} \, .
\end{align*}
At least formally, the evolution of $x_t:= X_{\mu_t}$ is then given by
\begin{align}	\label{eq:ODEpseudo}
\dot{x}_t(z) = v[ \mu_t(x_t(z))] \, ,
\end{align}
where $v[\mu] := f[\mu]/\mu$ is the Eulerian velocity.
The above mentioned works consider the discrete model given by the space of $N+1$ ordered particles $\mathcal K_N := \{ x = (x_0, \dots, x_N) \suchthat x_0 < \dots < x_N \} \subset \R^{N+1}$ and space-discrete evolutions obtained via a natural discretisation of \eqref{eq:ODEpseudo}. 
Folloeing \cite{difrancesco2015}, for compactly supported measures $\mu$, we set $x_j^N:= X_\mu(j/N)$ for $j=0, \dots, N$; then we notice that formally we also have $\mu_t(x_t(z))=\partial_z x_t(z)^{-1}$, so that we can use the discrete derivative of $x_t$ as a recostruction for the density. In the end we consider the system of ODEs 
\begin{align}	\label{eq:FTL}
\dot{x}_j^N(t) = v[R_j^N(t)] \, , \quad R_j^N(t) = \frac1{N (x_{j+1}^N(t) - x_j^N(t) )} \, , \quad  j = 0, \dots, N-1 \, , 
\end{align}
There is no canonical definition of the rightmost reconstructed density $R_N^N(t)$, since $x_N^N(t)$ have no particles in front. For example, in \cite{difrancesco2015} the authors set $R_N^N(t) = \arg \max v(\cdot)$, the so-called \textit{Follow-The-Leader scheme}. The name comes from the fact that the last particle $x^N_N$ moves with maximum velocity, whilst the other particles "see" the one in front of them, and they choose their velocity accordingly. 
It is then showed that, both the empirical and piecewise constant measures
\begin{align}	\label{eq:piecewise-empirical}
\E^N(x^N(t)):= \frac1{N+1} \sum_{j=0}^N \delta_{x_j^N(t)} , \quad 
	\cPC^N(x^N(t)):= \sum_{j=0}^{N-1} R_j^N(t) \mathbbm{1}_{\left[ x_j^N(t), x_{j+1}^N(t) \right)} 
\end{align}
converge as $N \to +\infty$ to the (unique) entropy solution of the associated conservation law.

The main goal of this work is to introduce a suitable transport distance at the discrete level, on the open cone $\cK_N$, which plays the role of a discrete approximation of the optimal transport distance $\btwoW$ on $\cP(\R)$, and prove a convergence result as $N \to \infty$; then we will extend this approach to the gradient flow structure introducing the JKO scheme at the discrete level. 
For every given concave mobility $m:\setR_+ \to \setR_+$ and $N \in \setN$, we define
\begin{align}
d_{2,m}^N(z,y)^2 := \inf
	\left\{ 
\int_0^1 \frac{1}{N+1}  \sum_{j=0}^N \frac{|\dot x_j(t)|^2}{\theta(R_j(t))} \de t \suchthat
x(0) = z, \, x(1) = y \, , \  x(t) \in \cK_N  \right\} \, , 
\end{align}
where $\theta(\rho):= m(\rho)/\rho$ and $R_j(t)$ as in \eqref{eq:FTL}. Our first contribution is to show that the discrete (pseudo)distances $d_{2,m}^N$ are a good approximation of the continuous ones $\btwoW$. We define $\bNtwoW: \cP(\R) \times  \cP(\R) \to \R_+ \cup \{ + \infty \}$ as 
$$\bNtwoW(\mu, \nu ) = \begin{cases} d_{2,m}^N(x,y) \qquad &\text{ if }\mu=\cE^N(x) \text{ and } \nu=\cE^N(y) \text{ for some }x,y \in \cK_N \\ +\infty &  \text{ if } \mu \notin \cE^N ( \cK_N) \text{ or } \nu \notin \cE^N ( \cK_N)  \end{cases}$$

\begin{center}
	\textit{First main result}:  the (pseudo)distances $\bNtwoW$   $\Gamma$-converge to $\btwoW$ as $N \to +\infty$, with respect to the vague and the classical $\bW_2$-topology.
\end{center}

Note that in the interesting case when m is assumed to have compact support, the definition of $\bNtwoW$ presents an additional singularity, which in particular makes this convergence less trivial. 

The study of the discrete distances opens the doors to possible applications to gradient flows approximation. In particular, at fixed $N \in \N$, for a given functional on $F_N:\cK_N \to \setR$ and initial value $x^{(0),N} \in \cK_N$, one can consider the associated JKO-scheme
\begin{align}
\label{eq:JKOnonlinearN}
x_\tau^{(n),N} \in \argmin_{x \in \cK_N} \left\{ F_N(x) + \frac1{2\tau}d_{2,m}^N \big( x, x_\tau^{(n-1),N} \big)^2 \right\}.
\end{align}
Assume that $F_N$ suitably converges to some continuous functional $F:\Prob(\setR) \to \setR$ and that the empirical measures $\mu^{(0),N}$ associated to the initial data $x^{(0),N}$ as in \eqref{eq:piecewise-empirical} converge to a limit $\mu^{(0)} \in \Prob(\setR)$. Similarly, denote by $\mu_\tau^{(n),N} = \E^N\big(x_\tau^{(n),N}\big)$ the empirical measures associated to $x_\tau^{(n),N}$.

\begin{center}
	\textit{Second main result}: for  $\tau>0$,  under suitable assumptions on $F_N \to F$ (see \eqref{i:A1}, \eqref{i:A2}, and \eqref{eq:cont_F_N_F}), the steps of the discrete JKO-scheme \eqref{eq:JKOnonlinearN} suitably converge to the steps of the continuous JKO-scheme \eqref{eq:JKOnonlinear},
	i.e. $\mu_\tau^{(n),N} \to \mu_\tau^{(n)}$ for every $n \in \N$, as $N \to +\infty$.
\end{center}
\smallskip

In particular, we believe that the variational point of view coming from the gradient-flow structure represents a powerful tool to handle evolution equations that might present deep mathematical difficulties, due to the lack of regularity and uniqueness.

\subsection{Sketch of the proofs}
For the convenience of the reader, we include a short sketch of the proofs of Theorem~\ref{thm:Gamma-conv} and Theorem~\ref{thm:JKO_q}, avoiding the technical details. In what follows, we denote by $\CE$ the set of solutions of the \textit{continuity equation} on $\R$, i.e. the curves $[0,1] \ni t \mapsto (\mu_t,\nu_t) \in \cP(\R) \times \cM(\R)$ which solves in the sense of distribution the equation
\begin{align*}
	\partial_t \mu_t + \partial_x \nu_t = 0 
		\, , \quad 
	t \in [0,1] \, .
\end{align*}

\subsubsection*{Sketch of the liminf inequality} 
Consider a sequence $\big\{ \big( x^N(0), x^N(1) \big) \big\}_N$, with $x^N(0), x^N(1) \in \cK_N$. Assume that the piecewise-constant measures $\mu_0^N:= \cPC^N(x^N(0))$, $\mu_1^N:= \cPC^N(x^N(1)) \in \cP(\setR)$ are vaguely (in duality with $C_c(\R)$) converging to $\mu_0,\mu_1 \in \cP(\setR)$ as $N \to \infty$. 
We want to show the \textit{liminf inequality}
\begin{align*}
	\liminf_{N \to \infty} d_{2,m}^N(x^N(0), x^N(1)) \geq \btwoW(\mu_0,\mu_1) \, .   
\end{align*}
Let $x^N(\cdot)$ be the discrete geodesic  between $x^N(0)$ and $x^N(1)$ with respect to $d_{2,m}^N$, viz. satisfying
\begin{align}
	\label{eq:sk_energy_bounds_geodes}
d_{2,m}^N(x^N(0), x^N(1)) = \int_0^1 \sum_{j=0}^N \frac{|\dot x_j^N(t)|^2}{\theta(R_j^N(t))} \de t
	\quad \text{and} \quad 
\sup_{N \in \setN}
\sup_{t \in [0,1]} 
	\sum_{j=0}^N 
		\frac{|\dot x_j^N(t)|^2}{\theta(R_j^N(t))} < \infty \,  .
\end{align}
We define $\mu_t^N:= \cPC^N(x^N(t))$ 
and consider the momentum fields given by
\begin{align}	
j_t^N(x):= \sum_{i=0}^{N-1} \dot{x}_i^N(t) R_i^N(t) \mathbbm{1}_{[x_i^N(t), x_{i+1}^N(t))} \, , 
\quad 
\nu_t^N := j_t^N \Leb^1 \, .
\end{align}
Intuitively, because the particles follows the evolution described by $\dot{x}_i^N$, we expect $(\mu_t^N, \nu_t^N)$ to be close to a solution to $\CE$. We show indeed, in the limit as $N \rightarrow +\infty$, that $\partial_t \mu_t^N + \partial_x j_t^N \to 0$ in the sense of distributions.

The second step is to provide a compactness result for (almost) solutions to $\CE$ satisfying energy bounds as \eqref{eq:sk_energy_bounds_geodes}, in the same spirit of \cite[Lemma 4.5]{dolbeault2012}: we prove that, up to a non-relabeled subsequence, $(\mu_t^N, \nu_t^N) \to (\mu_t,\nu_t) \in \CE$ in a weak topology (Proposition~\ref{prop:compactness_asymptCE}).

The final step is the comparison of the discrete and continuous energy functionals: by construction, we have that
\begin{align*}
	\int_\R \frac{|j_t^N|^2}{m(\rho_t^N)} \de x \leq \frac{N+1}{N} \sum_{j=0}^N \frac{|\dot x_j^N(t)|^2}{\theta(R_j^N(t))} 
		\, , \quad 
	t \in [0,1] \, .
\end{align*}
The proof of the liminf inequality is then concluded by \eqref{eq:sk_energy_bounds_geodes} and by means of lower semicontinuity properties of the continuous action \eqref{i:lsc}.

\subsubsection*{Sketch of the limsup inequality}
We fix $\mu_0,\mu_1 \in \cP_q(\R)$ two probability measures with $q$-th moment finite and such that $\btwoW(\mu_0,\mu_1) < \infty$. We pick $(\mu_t,\nu_t)\in\CE$ the unique $\btwoW$-geodesic between $\mu_0$ and $\mu_1$. The goal is to find a sequence $\{ (x^N(0),x^N(1) ) \}_N$ with $x^N(i) \in \cK_N$ such that we have the \textit{limsup inequality}
\begin{align*}
	\limsup_{N \to \infty} 	
		d_{2,m}^N(x^N(0), x^N(1)) 
	\leq
		\btwoW(\mu_0,\mu_1) \, .
\end{align*}
The idea is to first regularise the curve $(\mu_t,\nu_t)$ in order to obtain a solution of $\CE$ with compact support and smooth densities (w.r.t to the Lebesgue measure on $\R$), comparable energy, and whose density does not assume value at the boundary of the support of $m$ (which creates singularities in the energy functional). This is the content of Proposition~\ref{prop:approximation}. Therefore, without loss of generality, we assume here that $(\mu_t,\nu_t) \in \CE$ has such properties: 
\begin{enumerate}
	\item[i)] for all $t \in [0,1]$, $\mu_t$ and $\nu_t$ have compact and connected support, $\mu_t, \nu_t \ll \Leb^1$ and their respective densities densities $\rho_t$, $ j_t$ are $C^\infty$ in their support
	\item[ii)] there exist a compact set $K$ and two constants $0 < l \leq \tilde M < M$ such that 
	\[ \supp(\rho_t), \, \supp( j_t) \subset K  \quad \text{and} \quad  l  < \rho_t(x) < \tilde{M} \quad  \] 
	\text{for all  $x \in \supp(\rho_t)$ and $t \in [0,1]$.}
\end{enumerate}
Therefore, we can consider the flow map $\Psi_{v}$ associated with $v_t:= j_t / \rho_t$ and define for $t \in [0,1]$
\begin{align*}
	X_t(z):= \Psi_v(t,X_0(z)), \quad X_0 := X_{\mu_0}, \quad X_0 \in C^1([0,1]) ,
\end{align*}
which is nothing but the quantile function $X_t = X_{\mu_t}$ of $\mu_t$. The sought recovery sequence can then be obtained as
\begin{align*}
	 x_i^N(t) :=   X_t \left( \frac{i}{N} \right), \quad t \in [0,1], \; i \in \{ 0, ..., N \} \, .
\end{align*}
It is not hard to see that $\mu_t^N:= \cPC^N(x^N(t)) \to \mu_t$ in $\bW_q$ as $N \to \infty$. Moreover, by means of Jensen's inequality and by the regularity assumptions on $(\mu_t,\nu_t)$, one can show that the discrete energy of $x^N(t)$ converge to the one of $(\mu_t,\nu_t)$ as $N \to \infty$, which concludes the proof.

\subsubsection*{Sketch of the JKO-scheme convergence}
The convergence of the discrete JKO-schemes relies on the convergence of the distances. Indeed, under suitable assumptions on the functionals $F_N$, one can reduce the proof to the simpler case of $F_N \equiv 0$. In particular, it suffices to show that, for some $q>1$, the functionals 
\begin{align*}
	y^N \mapsto \cJ_N(y^N) :=d_{2,m}^N(x^N,y^N)^2
		\, , \quad 
	\text{for} \; x^N \in \cK_N 
		\; \text{s.t.} \;  
	\begin{cases}	\displaystyle
		\E^N(x^N) \xrightarrow{\bW_q} \mu^{(0)} \in \cP(\R) , 
	\\	\displaystyle 
		\sup_{N \in \N} \bE_{\E^N(x^N)} |x|^2 < \infty ,
	\end{cases}
\end{align*}
are \textit{equicoercive} and \textit{$\Gamma$-converging} to the continuous one $\mu \mapsto \cJ(\mu):= \btwoW(\mu^{(0)}, \mu)$.

In order to have equicoercivity for $\seq{\cJ_N}_N$, and because we do not want to enforce any coercivity for the functionals $\seq{F_N}_N$ (due to the applications we have in mind), we work with $q<2$. Indeed, using that balls in $\bW_2$ are pre-compact in $\bW_q$, for every $q<2$ (but not $q=2$!), it is not hard to prove the sought equicoercivity. On the other hand, working with $q<2$ makes the $\Gamma$-convergence less trivial. 

Let us say we have $\big(x^N\big)_N$, $\mu^{(0)}$, $\mu$ given and we want to find a recovery sequence $\big(y^N\big)_N$ for $\seq{\cJ_N}_N$. A first n\"aive attempt would be to apply Theorem~\ref{thm:Gamma-conv}, which ensures the existence of a recovery sequence $\{(\tilde x^N,\tilde y^N)\}_N$ such that $d_{2,m}^N(\tilde x^N, \tilde y^N) \to \btwoW(\mu^{(0)},\mu)$, and set $y^N:=\tilde y^N$. Nonetheless, having by assumption that $\E^N(x^N) \to \mu^{(0)}$ in $\bW_q$ (\textit{not} in $\bW_2$), a priori we are not able to show that $d_{2,m}^N(x^N,\tilde x^N) \to 0$.

However, thanks to the fact that $\big(x^N\big)_N$ has uniformly bounded second moments, we are able to provide the existence of the sought recovery sequence in $\bW_q$, for $q<2$. This relies on an quantitative regularisation result (Proposition~\ref{prop:approximation}, in particular $(iii)$) and a careful construction of the recovery sequence, based on Theorem~\ref{thm:Gamma-conv} together with a delicate localisation argument (see proof of Lemma~\ref{prop:limsup_JKO}).

\subsection{Related literature}
	\label{sec:literature}
The Wasserstein gradient-flow structure is a common feature of several evolution equations, including Fokker--Planck equations, porous-medium \cite{otto2001geometry}, fourth-order evolutions \cite{matthes2009,loibl2016}, and more generally can describe diffusion, aggregation, and advection effects, see \cite{Santambrogio:2017} for a general discussion about various possible combinations. 

The JKO scheme as a numerical tool to approximate PDEs has recently been the focus of several works, such as porous medium equations  \cite{benamou2016} and the Patlak--Keller--Segel model \cite{blanchet2008}. In the setting of non-linear mobilities, this variational strcture is exploited in a couple of recent works, in the case of scalar Burger equations \cite{gigli2013} and thin-films evolutions \cite{lisiniMatthes2012}.

In the last decade, a lot of work have been devoted to the study of discrete approximation of optimal transport distances \cite{GiMa13,trillos,GlKoMa18,gladbach2020,gladbach2021} and of the associated gradient flows \cite{disser2015gradient,Forkert2020}. Nonetheless, these works deal with a discretisation of the underlying space (here $\R$), and do not work with deterministic particle methods.

The notion of generalised Wasserstein distances has been introduced in \cite{dolbeault2012,lisiniMarigonda2010} and the study of the corresponding gradient-flow structures found applications to the study of non-linear cross-diffusion systems \cite{burger2010} and fourth-order equations, such as Cahn-Hillard and thin film equations \cite{lisiniMatthes2012},  \cite{zinsl2017}. See also 
\cite{carrillo2010} for the study of convexity properties of functionals on the space of measures with respect to such transport distances.

On the other hand, the validation of the continuum evolutive equations is classically obtained proving the convergence of systems of interacting particles.
The literature involving probabilistic methods is extremely rich and quite well-understood, see e.g. \cite{demasi1991}, \cite{stroock1979}.
A first rigorous fully determistic approximation result for non-linear conservation laws has been obtained in \cite{difrancesco2015}. More in general, many-particle limit results for scalar evolutions with non-local interaction  have been obtained for both linear \cite{daneri2020deterministic,fagioli2020opinion} and non-linear mobilities \cite{diFrancesco2019convergence,difrancescoFagioli2019,fagioli2018,radici2021entropy,Fagioli-Tse:2022}. In particular, these works concern the validity of the top, blue arrow in Figure~\ref{fig:1}.

In all the above mentioned deterministic methods, the particle system is obtained as a discretised version of the Lagrangian formulation of the evolution, using the pseudo-inverse formalism.
We emphasise that \cite{Fagioli-Tse:2022} shows the convergence of the particle systems to the corresponding entropy solution, working directly at the level of the gradient-flow structures at the continuous-time scale (but with no analysis of the behavior of the discrete distances nor the time-discrete scheme).

\subsection*{Organisation of the paper}
In Section~\ref{sec:setting} and \ref{sec:main} we introduce the setting of the problem and state our main results. Section~\ref{sec:approx} includes some preliminary results concerning the approximation of finite-energy solutions to the continuity equation in $\R$, whereas in Section~\ref{sec:proof_dist} and \ref{sec:JKO} we prove our two main results, i.e. the convergence of the discrete distances and of the gradient-flow schemes. 

\subsection*{Notation}
\noindent
For $d,m \in \N$, $\Omega \subset \R^m$, we denote by $\cM(\Omega)$ (resp. $\cP(\Omega)$) the set of all signed, locally finite Radon measures (resp. probability measures) on $\Omega$. Also, $C(\Omega)$, $C_b(\Omega)$, $C_c(\Omega)$ denote respectively the space of continuous, bounded, and compactly supported functions.

For a given sequence of measures $\seq{\nu_n,\nu}_n\subset \cM(\Omega)$, we write
\begin{enumerate}[(i)]
	\item $\nu_n \to \nu$ \textit{vaguely} in $\cM(\Omega)$ if $\int \vphi \de \nu_n \to \int \vphi \de \nu$, for every $\vphi \in C_c(\Omega)$.
	\item $\nu_n \to \nu$ \textit{narrowly} in $\cM(\Omega)$ if $\int \vphi \de \nu_n \to \int \vphi \de \nu$, for every $\vphi \in C_b(\Omega)$.
\end{enumerate}
We also adopt the following notation:
\begin{itemize}
	\item For $\mu \in \cP(\Omega)$, $s \in \R_+$, set  $\bE_\mu |x|^s:= \int |x|^s \de \mu(x)$ ($s$-th moment).
	\item For $N \in \N$, $x \in \R^{N+1}$, $p \geq 1$, we write $\norm{x}_{\ell_{p,N}}:=\frac{1}{N+1} \sum_{i=0}^N |x_i|^p$.
	\item $\Leb^d$ denotes the Lebesgue measure on $\R^d$, $d \in \N$.
	\item For $a,b \in \R$, $a \vee b = \max\{a,b\}$.
\end{itemize}

\section{The discrete and continuous settings}
\label{sec:setting}
In this section, we introduce the general framework of the problem, both in the continuous and discrete setting. 
\subsection{Optimal transport in $\setR$ with nonlinear mobility}

We start recalling the definition of solutions of the continuity equation (as in \cite{dolbeault2012}, \cite{lisiniMarigonda2010}), namely the equation
$
	\partial_t \mu_t + \partial_x \nu_t = 0, 
$ 
in $(0,1) \times \setR$, 
where $\mu_t,\nu_t$ are Borel families of measures in $\cP(\setR)$ and $\cM(\R)$ respectively, and satisfying
the equation in the sense of distributions, namely, $\forall \xi \in C_c^1((0,1) \times \R )$,
\begin{align}	\label{eq:CEdist}
	\int_0^1 \int_{\setR} \partial_t \xi(t,x) \de \mu_t(x) \de t + \int_0^1 \int_{\setR} \partial_x \xi(t,x) \de \nu_t(x) \de t = 0 \, .
\end{align}
For a given sequence of measures $\nu_t$ we consider the corresponding measure $\bfnu:=\int_0^1 \nu_t \de t \in \cM((0,1) \times \R)$ defined via disintegration
\begin{align}\label{eq:bnu}
	<\bfnu,\xi> = \int_0^1 \Big( \int_\setR \xi(t,x) \de \nu_t(x) \Big) \de t, \quad \forall \xi \in C_c((0,1) \times \R ).
\end{align} 
Similarly, we write $\bfmu:=\int_0^1 \mu_t \de t \in \cP((0,1) \times \R)$. 
\begin{definition}[Solution of continuity equation, \cite{lisiniMarigonda2010}]
	\label{def:CE}
We denote by $\CE$ the set of curves $\seq{\mu_t,\nu_t}_{t \in [0,1]}$ such that
\begin{enumerate}
	\item $t \mapsto \mu_t$ is vaguely-continuous in $\cP(\setR)$.
	\item $\seq{\nu_t}_t$ is a Borel family of finite Radon measures on $\R$.
	\item $(\bfmu,\bfnu)$ satisfies \eqref{eq:CEdist}.
\end{enumerate}
In this case we write in short 
	$
		\partial_t \mu_t + \partial_x \nu_t = 0
	$ and 
	$
		(\mu_t,\nu_t) \in \CE 
	$, or $(\bfmu, \bfnu) \in \CE$.
Whenever we want to emphasis the dependence on the initial and final datum, we write $(\mu_t, \nu_t) \in \CE(\mu_0,\mu_1)$. 
\end{definition}
\noindent
We now fix the notation concerning the action functionals. See \cite[Section 2.3]{lisiniMarigonda2010}.

\begin{definition}[Nonlinear mobility]
\label{def:theta}
Given a maximal density $M>0$, a \textit{nonlinear mobility} is a continuous concave function $m:[0,M] \to  [0,+\infty)$ such that $m(0)=0$; sometimes, if needed we extend the domain of $m$ to $[0,+\infty)$ letting $m(\rho)=-\infty$ for $\rho>M$. We moreover define
	\begin{align*}
		\theta:[0,M] \to  [0,+\infty)
			\, , \quad 
		\theta(\rho) :=
	 \begin{cases}
	 	\frac{m(\rho)}{\rho} &\text{if} \ \rho >0  , \\
	 {	\displaystyle
	 	\lim_{\rho \downarrow 0} } \, \frac{m(\rho)}{\rho} 
	 		&\text{if} \ \rho =0. 
	 \end{cases}  
	\end{align*}
	Notice that $\theta$ is an (extended) continuous function in $[0,M]$; in this work, we often assume that $\theta$ is also bounded. Note that by concavity of $m$, this is equivalent to the superdifferential of $m$ being not empty in $\rho=0$. \BLK
\end{definition}

\begin{remark}[Monotonicity of $\theta$]
\label{rem:monotonicity_theta}
	Note that by concavity of $m$, it follows that $\theta$ is nonincreasing on $(0,M]$. Indeed, for $0 < \rho_1 <\rho_2 \leq M$, we have that
	\begin{align*}
		\theta(\rho_2)-\theta(\rho_1) 
			=
		\frac1{\rho_2}
			\Big[
				\big(m(\rho_2) - m(\rho_1)\big) - \frac{m(\rho_1)}{\rho_1}(\rho_2 - \rho_1)
			\Big]
		\leq 0
	\end{align*}
	by concavity of $m$ on $[0,M]$ and $m(0)=0$. In particular, $ \sup \theta  = \theta(0)$.
\end{remark}

\begin{definition}[Action functionals]
		\label{def:action_density_functions}
Fix a maximal density $M>0$ and $p>1$. We then consider $\phi_{p,m}:\setR \times \setR \rightarrow [0,+\infty]$ defined as
\begin{align*}
	\phi_{p,m}(\rho,j) = 
	\begin{cases} \displaystyle
		\frac{|j|^p}{m(\rho)^{p-1}}, & \If \ m(\rho)>0 , \\
		0 & \If \  m(\rho)=0=j, \\ 
		+\infty, & \If \emph{either}  \ m(\rho)=0  \, ,  j \neq 0 \emph{ or } \rho>M \, ,
	\end{cases}
\end{align*}
where $m$ is a nonlinear mobility with maximal density $M$. Note that $\phi_{p,m}$ is a lower semicontinuous, nonnegative, proper, convex function, with $\inte( \dom (\phi_{p,m})) = (0,M)\times \R$. 
For every $\mu \in \cP(\R)$ and $\nu \in \cM(\setR)$, we define the 
\textit{action functional} 
\begin{align}
\label{eq:def_Phi}
	\Phi_{p,m}(\mu,\nu) :=
	\begin{cases}
\displaystyle
		\int_{\setR} \phi_{p,m}(\rho,j) \de x
			&\text{if} \quad \de \mu = \rho \de x \, , \ \de \nu= j \de x \, , 
	\\
		+\infty 
			&\text{otherwise} \, .
	\end{cases}
\end{align}
\end{definition}

\begin{remark}	\label{rmk:finiteaction_abscont}
As observed in \cite[Proposition 2.3]{lisiniMarigonda2010}, if $\mu$, $\nu$ are such that $\Phi_{p,m}(\mu,\nu)<\infty$ (in particular they are absolutely continuous w.r.t. d$x$), then $\de \mu = \rho \de x$ with $0 \leq \rho(\cdot) \leq M$.
\end{remark}

\begin{definition}[$m$-Wasserstein distances] 
Let $\phi_{p,m}$, $\Phi_{p,m}$ be as in Definition~\ref{def:action_density_functions}. 
For any $\mu_0, \mu_1 \in \Prob_p(\setR)$, we define
\begin{align}
		\label{eq:defWphi}
	\bpW(\mu_0,\mu_1) := \inf \left\{ \int_0^1 \Phi_{p,m}(\mu_t,\nu_t) \de t \suchthat (\mu_t,\nu_t) \in \text{CE}(\mu_0,\mu_1) \right\}^{\frac1p}.
\end{align}
Note that $\bpW$ is only a pseudo-distance, given that it can assume the value $+\infty$. If $m(\rho)=\rho$, then we recover the usual $p$-Wasserstein distance $\bW_p$.
\end{definition}

\begin{remark}[Time regularity of finite energy curves]
		\label{rem:time_reg_finite_energy}
	Curves of measures with finite energy enjoy good regularity properties in the time-variable, if $\theta $ is bounded. Precisely,  if $(\mu_t,\nu_t)\in\bCE$ and 
	\begin{align*}
		 E:=\int_0^1 \Phi_{p,m}(\mu_t,\nu_t) \de t < \infty \, .
	\end{align*}
	Then $t \mapsto \mu_t$ is Lipschitz-continuous with respect to the $\bW_p$ distance on $\Prob_p(
	\setR)$. Indeed, by the homogeneity of $\phi$ in the second variable, 
	\begin{align*}
		\bW_p(\mu_s, \mu_t)^p
		\leq \| \theta \|_\infty \bpW(\mu_s,\mu_t)^p
		\leq \| \theta \|_\infty  |t-s|^p \int_0^1 \Phi_{p,m}(\mu_h, \nu_h) \de h 
		\leq E \norm{\theta}_\infty |t-s|^p \, ,
	\end{align*}
	as claimed. In particular, by Prokhorov's theorem, the family of probability measures $\seq{\mu_t}_{t \in [0,1]}$ is tight in $\Prob(\setR)$.
\end{remark}

\subsubsection{Auxiliary results}
	\label{subsec:auxiliary}
In this section we recall some well-known properties of the actions and the generalised distances. We refer to \cite{dolbeault2012} and \cite[Section~3\&4]{lisiniMarigonda2010} for details. Let $\phi_{p,m}$ be as in Definition~\ref{def:action_density_functions}.

\begin{enumerate}[(I)]
	\item \label{i:lsc}
	\textit{Lower semicontinuity}. \
	For every $\seq{\mu_n}_n \in \cP(\R)$, $\seq{\nu_n}_n\subset \cM(\R)$ vaguely converging to $\mu$, $\nu$ respectively, then $\displaystyle\liminf_{n \to \infty} \Phi_{p,m}(\mu_n, \nu_n) \geq \Phi_{p,m}(\mu,\nu)$.
	\item \label{i:convolution}
	\textit{Convolution}. \
	Let $h \in C_b(\R \times \R)$ be a positive function with $\int h(x) \de x = 1$. For $[0,1] \ni t \mapsto (\mu_t,\nu_t) \in \cP(\R) \times \cM(\R)$ a curve of measures, we consider their natural extension on $\R$ by setting
	\begin{align*}
		\mu_t =  \mu_0 
		\ \ \text{if} \ \ t< 0
			\, , \quad 
		\mu_t =  \mu_1 
		\ \ \text{if} \ \  t> 1
			\, , \quad 
		\nu_t = 0 \ \ \text{if} \ \ t \in (0,1)^c \, .
	\end{align*}
	Let $(\bfmu, \bfnu)$ as in \eqref{eq:bnu}; define $\hat\bfmu:= h * \bfmu$ and $\hat \bfnu := h * \bfnu$, where $h * \bfnu$ denotes the measure which is absolutely continuous with respect to the Lebesgue measure on $\R \times \R$ with density $\hat j(t,x) := \int h(t-s, x-y) \de \bfnu(s,y)$.  Note that if $(\mu_t, \nu_t) \in \CE$, then $(\hat\mu_t, \hat\nu_t) \in \CE$ as well. Moreover, by $\phi \geq 0$ and $\phi(\cdot, 0)=0$, we have that \begin{align*}
		\int_0^1 
		\Phi_{p,m}(\hat \mu_t,\hat \nu_t) \de t
	\leq	
		\int_0^1 \Phi_{p,m}(\mu_t,\nu_t) \de t  \, .
	\end{align*}
	\item  \label{i:geodesics_convexity}
	\textit{$\bpW$-geodesics}. \
	Fix $\mu_0$, $\mu_1 \in \cP(\R)$ such that $\bpW(\mu_0,\mu_1) < \infty$. Then there exists $(\mu_t,\nu_t) \in \CE(\mu_0,\mu_1)$ such that $\bpW(\mu_0,\mu_1) = \int_0^1 \Phi_{p,m}(\mu_t,\nu_t) \de t$. Furthermore, $\mu_t$ is a constant speed geodesic for $\bpW$, i.e. it satisfies $\bpW(\mu_t,\mu_s) = |t-s| \bpW(\mu_0,\mu_1)$, for all $s,t \in [0,1]$. 

\end{enumerate}

\subsection{The discrete transport problem}
\noindent
In this section we introduce the discrete setting. Recall that $M \in (0,+\infty)$ is a fixed positive number. For any $N \in \setN$, we consider the (closed) cone $\cK_N$ in $\setR^{N+1}$ given by
\begin{align*}
	\cK_N:= \left\{ x = (x_0, ..., x_N) \in \setR^{N+1} \suchthat  x_{i+1} - x_i \geq \frac1{NM} \, , \; \forall i \right\} \subset \setR^{N+1}.
\end{align*}

For simplicity, we omit the dependence on $M$ in $\cK_N$. We consistently adopt the following notation: for $x \in \cK_N$,  we denote by
\begin{align}
\label{eq:notation_R_i}
	\Delta x_i := x_{i+1}- x_i, \quad R_i^N(x):= \frac{1}{N \Delta x_i} \leq  M  \, , 
\end{align}
for $i=0,...,N-1$, and $\rho^* \in \R_+$ a nonnegative number. The definition of the density reconstruction for $i=N$ admits no canonical definition (there is no particle in front of $x_N$), and several choice are possible. We define $R_N^N(x):= \rho_N^*(x)$, where  $\rho_N^*: \cK_N \to \R_+$ is a family of continuous functions satisfying
\begin{align}
\label{eq:ass_rho*}
	\sup_{N \in \N}  \| \rho_N^* \|_\infty <M \, .
\end{align}  
\begin{remark}[Examples of $\rho_N^*$]
	A simple choice is to set $\rho_N^*(x) \equiv \rho^*:= \arg \max \theta$, then the last particle motion is not influenced by the rest of the particles (\textit{follow-the-leader} scheme, see e.g. \cite{difrancesco2015}). Another option could be to "look back" to the previous particle, i.e. $\rho_N(x)= R_{N-1}^N(x)$.
\end{remark}

\begin{definition}[Discrete distances]
	\label{def:discrete_phi_action}
Fix $N \in \setN$, $p>1$, $M \in \R_+$. Let $\phi_{p,m}$ be as  in Definition \ref{def:action_density_functions}. We define the corresponding \textit{discrete action functional} $\Phi_{p,m}^N:\TN \to \setR^+ \cup \{ +\infty \}$,
\begin{align}	\label{eq:def_PhiN}
		\Phi_{p,m}^N(x, v) := \frac{1}{N+1} \sum_{i=0}^{N} (R_i^N(x))^{p-1}\phi_{p,m}(R_i^N(x),v_i^N)
			=
		\frac1{N+1}
			\sum_{i=0}^N
				\frac{|v_i|^p}{\theta(R_i^N(x))^{p-1}} \, , 
\end{align}
for any $x \in \cK_N$ and $v\in \xTN{x}$. Here $\TN \approx \cK_N \times \setR^{N+1}$ denotes the tangent bundle of the cone $\cK_N$.
Furthermore, we define the \textit{discrete  $\phi_{p,m}$-distance} between $x,y \in \cK_N$ as
\begin{align}	
\label{eq:defdN_vphi}
d_{p,m}^N(x,y)^p:=	 
\inf_{x \in W^{1,p}(0,1 ; \cK_N)} 
\left\{ 
\int_0^1 \Phi_{p,m}^N(x(t), \dot{x}(t))  \de t \suchthat x(0) = x, \; x(1) = y 
\right\}, 
\end{align}
where $W^{1,p}(0,1 ; \cK_N)$ denotes the space of $p$-Sobolev functions from $(0,1)$ to $\cK_N$.
Note that $d_{p,m}^N$ is only a pseudo-distance, given that it can assume the value $+\infty$. 
\end{definition}

\begin{remark}
	For the sake of simplicity, we omit the dependence of $\Phi_{p,m}^N$ on $\rho^*$. As we are going to see, the exact value chosen plays no role in our main results, due to the fact we consider limits as $N \to \infty$. If $\rho^*(x)\geq M$ (and thus $\theta(R_N^N(x))=0$), then $\phi_{p,m}(x,v)<\infty$ forces $v_N =0$.
\end{remark}

\begin{example}[Riemannian case]
In the case $p=2$, $d_{2,m}^N$ is a Riemannian distance induced by the Riemannian metric on the cone $\cK_N$ given by 
\begin{align}	\label{eq:defgN}
	g_N \left( v, w \right)_x := \frac{1}{N+1} \sum_{i=0}^{N} \frac{v_i w_i}{\theta(R_i^N(x))}, \quad \forall x \in \cK_N, \; \forall v,w \in \xTN{x} \, .
\end{align}
\end{example}

Now we prove the existence of geodesics in $\cK_N$, which follows by standard compactness arguments, together with a careful analysis close to the singularity of $\Phi_{p,m}^N$. 
\begin{lemma}[Existence of discrete geodesics]
		\label{lemma:geodesics_discrete}
	Let $\theta \in \emph{L}^\infty(0,M)$ and $p>1$. For every $z,y\in \cK_N$, there exists $x \in W^{1,p}(0,1 ; \cK_N)$ such that $x(0)= z$, $x(1) = y$ and
	\begin{align*}
		d_{p,m}^N(z,y)^p = \int_0^1 \Phi_{p,m}^N
			( x(t), \dot x(t) ) \de t 
		=  \Phi_{p,m}^N
		( x(s), \dot x(s) ) \, , 
		\quad 
			\forall s \in [0,1] \, .
	\end{align*}
\end{lemma}
\begin{proof}
For $z,y \in \cK_N$, pick a minimising sequence $\seq{x^n}_n \subset W^{1,p}(0,1 ; \cK_N)$, i.e.  
\begin{align*}
	\lim_{n \to \infty} 
		\int_0^1 \Phi_{p,m}^N(x^n(t), \dot x^n(t)) \de t = d_{p,m}^N(z,y)^p
	\, , \quad 
		x^n(0) = z 
	\, , \; 
		x^n(1) = y \, .
\end{align*}
With no loss of generality, we assume $D:= \sup_n \int_0^1 \Phi_{p,m}^N(x^n, \dot x^n) \de t < \infty$. In particular
\begin{align*} 
		\sup_{n \in \N} \int_0^1 \norm{\dot x^n(t)}_{\ell_{p,N}}^p \de t \leq D \norm{\theta}_\infty^{p-1} < \infty\, .
\end{align*}

In particular, up to (a non-relabeled) subsequence, $x^n \to  x$ weakly in $ W^{1,p}(0,1 ; \cK_N)$ and strongly in $C(0,1 ; \cK_N)$ (by Sobolev embedding, thanks to $p>1$) as $n \to \infty$. To show that $x$ is indeed a minimiser, we show the lower semicontinuity of the discrete energies with respect to such convergence.

First we assume that $\theta(M)>0$ and in particular $\theta$ is continuous and bounded from below. Now, since $R_i^N(x^n)$ is converging uniformly to $R_i^N(x)$, we also have that $\theta(R_i^N(x^n))^{1-p}$ will converge uniformly to $\theta(R_i^N(x))^{1-p}$. Since $\dot{x}^n \weakto \dot{x}$ weakly in $L^p$, we then can conclude, by classical semicontinuity arguments, that
$$ \liminf_{n \to \infty} \int_0^1 \frac{ |\dot{x}^n_i(t)|^p}{\theta(R_i^N(x^n(t)))^{p-1}}  \de t \geq  \int_0^1 \frac{ |\dot{x}_i(t)|^p}{\theta(R_i^N(x(t)))^{p-1}}\de t.$$
Summing up over $i=0, \ldots, N$ and then dividing by $N+1$ we get the actual semicontinuity of the energy $\liminf_{n \to \infty} \int_0^1 \Phi_{p,m}^N(x^n, \dot x^n) \de t \geq\int_0^1 \Phi_{p,m}^N(x, \dot x) \de t$, hence $x$ is a minimizer.

In order to deal with the general case we consider for $\lambda >1$ a relaxed mobility $m^{\lambda}(\rho)=\lambda m(\rho/\lambda)$. Of course we have $\theta^{\lambda}(\rho)=\theta(\rho/\lambda)$. Using the monotonicity of $\theta$ we have that $\Phi_{p,m^{\lambda}}^N(x^n, \dot x^n) \leq \Phi_{p,m}^N(x^n, \dot x^n)$. Moreover, since $\theta^{\lambda}(M)=\theta(M/\lambda)>0$ we can apply the previous result and we can get
$$\liminf_{n \to \infty} \int_0^1 \Phi_{p,m}^N(x^n, \dot x^n) \de t \geq \liminf_{n \to \infty} \int_0^1 \Phi_{p,m^{\lambda}}^N(x^n, \dot x^n) \de t \geq \int_0^1 \Phi_{p,m^{\lambda}}^N(x, \dot x) \de t.$$
By monotone convergence we then conclude since $\Phi_{p,m^{\lambda}}^N(x, \dot x) \uparrow\Phi_{p,m}^N(x, \dot x)$.

The fact that $s \mapsto \Phi_{p,m}^N
( x(s), \dot x(s) )$ is constant along discrete geodesics follows by classical arguments and is consequence of the $p$-homogeneity of $\phi$. Let us argue by contradiction that it is not true; in particular by Jensen we have $\int_0^1 \Phi_{p,m}^N( x(s), \dot x(s) ) \,ds > \left(\int_0^1 \Phi_{p,m}^N( x(s), \dot x(s) )^{1/p} \,ds \right)^p$. Now let us consider any $C^1$ increasing reparametrization $\xi:[0,1] \to [0,1]$, and let $y(t)=x(\xi(t))$; we have of course $y(0)=0$ and $y(1)=1$. Moreover
\begin{align*}\int_0^1 \Phi_{p,m}^N(y(t), \dot y (t)) \,dt &= \int_0^1 \Phi_{p,m}^N(x(\xi(t)),\xi'(t) \dot x (\xi(t))) \,dt  \\ & \geq \left( \int_0^1 \Phi_{p,m}^N(x(\xi(t)), \xi'(t)\dot x (\xi(t)))^{1/p} \,dt\right)^p\\
 &= \left( \int_0^1 \xi'(t) \Phi_{p,m}^N(x(\xi(t)), \dot x (\xi(t)))^{1/p} \,dt\right)^p \\ &= \left( \int_0^1 \Phi_{p,m}^N(x(s),\dot x (s))^{1/p} \,dt\right)^p.
\end{align*}
We can have equality as long as $\xi'(\xi^{-1}(s))=(\xi^{-1})'(s)$ is proportional to $\Phi_{p,m}^N( x(s), \dot x(s) )^{1/p}$. While this is not guaranteed to happen for $\xi \in C^1$ we can use a sufficiently good approximation in order to get $\int_0^1 \Phi_{p,m}^N(y(t), \dot y (t)) \,dt <\int_0^1 \Phi_{p,m}^N( x(s), \dot x(s) ) \,ds $, which is a contradiction.

\end{proof}

In order to compare the discrete and the continuous setting, we introduce suitable embedding maps from the cones $\cK_N$ to the space of continuous measures $\Prob(\setR)$.

\medskip
\noindent
\textit{Embedding via piecewise-constant measures.} \
For any $N \in \setN$, we consider the map
\begin{align}
		\label{eq:defPC^N}
	\cPC^N: \cK_N \to \Prob(\setR)
		\, , \quad 
	\cPC^N(x) = \sum_{i=0}^{N-1} R_i^N(x) \mathbbm{1}_{\left[  x_i, x_{i+1}\right]}(\cdot) \Leb^1 \, ,
\end{align}
where $\Leb^1$ denotes the Lebesgue measure on $\setR$. Note that, for every $x \in \cK_N$, the measures $\cPC^N(x)$ are absolutely continuous with respect to $\Leb^1$. Moreover, $\cPC^N(x)$ charges every interval $[x_i,x_{i+1}]$ with equal mass $\frac1N$. We shall use the shorthand notation $\cPC^N:= \cPC^N(\cK_N) \subset \Prob(\setR)$.

\medskip
\noindent
\textit{Embedding via empirical measures.} \
For any $N \in \setN$, we consider the map
\begin{align}
		\label{eq:defE^N}
	\E^N: \cK_N \to \Prob(\setR)
		\, , \quad 
	\E^N(x) = \frac{1}{N+1} \sum_{i=0}^N \delta_{x_i} \, .
\end{align}

Note that the image of $\E^N$ consists of atomic measures over the real line. In particular, $\E^N(x)$ is singular with respect to the Lebesgue measure. We adopt the shorthand notation $\E^N:= \E^N(\cK_N) \subset \Prob(\setR)$.

\medskip
It's easy to see that $\E^N$ and $\cPC^N$ are injective. Both embeddings maps play an important role in linking the discrete and  continuous framework. As we are going to see in Lemma \ref{lemma:weak_limits}, the two maps turn out to be equivalent at the level of the convergence, as $N \to \infty$, with respect to a broad class of weak topologies on $\Prob(\setR)$.

On one hand, the $\cPC^N$-embedding appears to be more natural in the discretisation of the continuous distances $\bpW$, for two reasons: firstly, the image measure is always absolutely continuous, as every measure in the domain of $\bpW$. Secondly, we note that, for $x \in \cK_N$ and $v \in T_x \cK_N$, one has that
\begin{align}	\label{eq:A<AphiN}
\Phi_{p,m}(\mu^N, \nu^N) \leq \frac{N+1}{N} \Phi_{p,m}^N(x,v)
\end{align}
for $\mu^N=\cPC^N(x)$ and $	\nu^N := 
		\sum_{i=0}^{N-1} 
			v_i R_i^N(x) \mathbbm{1}_{\left[  x_i, x_{i+1}\right]}(\cdot) \Leb^1 \in \cM(\R)$.

On the other hand, the $\E^N$-embedding enjoys good properties at the level of the $p$-moments. Indeed, by constructions, the $p$-th moments of $\E^N(x)$ coincides with 
$
		 \| x \|_{\ell_{p,N+1}}^p
$,
thus properties in the $p$-Wasserstein space can be recast in terms of the behavior in the space $(\cK_N, \| \cdot \|_{\ell_{p,N}})$.

Depending on the choice of the embedding, we can define the corresponding embedded distance at the continuous level.

\begin{definition}[Discrete distances - embedding]	\label{def:discrete_distance}
Fix $N \in \setN$, and let $\phi_{p,m}$ be an energy density as in Definition~\ref{def:action_density_functions}. Define the pseudo-distances $\bNW$ and $\bNhW$ on $\Prob(\setR)$ as 
\begin{align*}
	\bNW (\mu_0, \mu_1) 
	&:= 
	\begin{cases}
		d_{p,m}^N(x,y), &\text{if }\mu_0 = \cPC^N(x), \; \mu_1=\cPC^N(y), \; x,y \in \cK_N \, , \\
		+\infty, &\text{otherwise}  \, ,
	\end{cases} \\
	\bNhW (\mu_0, \mu_1) 
		&:= 
	\begin{cases}
	d_{p,m}^N(x,y), &\text{if }\mu_0 = \E^N(x), \; \mu_1=\E^N(y), \; x,y \in \cK_N \, , \\
	+\infty, &\text{otherwise} \, ,
	\end{cases}
\end{align*}
where $d_{p,m}^N$ is the discrete pseudo-distance defined in \eqref{eq:defdN_vphi}.
\end{definition}

\subsection{The minimising movements scheme}
Let $\phi_{2,m}$ be a density function as in Definition~\ref{def:action_density_functions}.
 Note that in this quadratic case, $\bW_{2,m}^N$ is induced by a Riemannian metric \eqref{eq:defgN} on $\cK_N$, denoted by $g_N$. In this section, we introduce the discrete-in-time schemes (both in $\cK_N$ and $\R$) for the gradient flow of an energy functional with respect to $\btwoW$, $\bW_{2,m}^N$, in the spirit of \cite{jordan1998variational}.

\begin{definition}[Continuous JKO-scheme]
	\label{def:discrete_JKO}
	For a given $F:\Prob(\setR) \to (-\infty, +\infty]$ and $\tau >0$, we iteratively define for $n \in \N$
	\begin{align}	\tag{JKO$_\tau$}
		\label{eq:JKOtau}
		\mu^{(0)} \in \Prob(\setR), \quad \mu_\tau^{(n)} \in \text{argmin}_\mu \left\{ F(\mu) + \frac{1}{2\tau} \btwoW \big(\mu,\mu_\tau^{(n-1)}\big)^2 \right\}.
	\end{align}
\end{definition}
Formally, in the limit as  $\tau \to 0$, one formally obtains the PDE
\begin{align*}
	\partial_t \mu_t - \partial_x \left(m(\mu_t) \partial_x \left( \text{D}F(\mu_t ) \right) \right)=0, \quad \mu_0 := \mu^{(0)} \in \Prob(\setR),
\end{align*}
where D$F$ denotes the first variation of $F$ (i.e. $($D$F(\rho), \omega) = \lim_{\eps \to 0} \eps^{-1} (F(\rho+\eps w)- F(\rho))$). This equation is the (formal) gradient flow of $F$ with respect to the generalised Wasserstein distance $\btwoW$.
The corresponding discrete counterpart is given by the following definition. 
\begin{definition}[Discrete JKO-schemes]	\label{def:discrete_JKO}
	For a given $N \in \setN$, $F_N:\Prob(\setR) \to (-\infty, +\infty]$, $\tau >0$, we iteratively define for $n \in \N$
	\begin{align}	\tag{JKO$_{\tau,N}$}
		\label{eq:JKOtauN}
		\mu_N^{(0)} \in \E^N, \quad \mu_{\tau,N}^{(n)} \in \argmin_{\mu \in \Prob(\setR)} \left\{ F_N(\mu) + \frac{1}{2\tau} \bNtwoW\big(\mu,\mu_{\tau,N}^{(n-1)}\big)^2 \right\}	\, , 
	\end{align}
	where $\E^N$ denotes the set of empirical measures as given in \eqref{eq:defE^N}.
\end{definition}

\medskip
\noindent
\underline{Assumptions}. In order to ensure well-posedness, we need the following assumptions. 

\begin{enumerate}[(A)]
	\item\label{i:A1} 
	There exist $C,D \in \setR^+$, $s \in (0,2)$, such that
	\begin{align}	\label{eq:lowercontrolF-pmoments}
		F(\mu) \geq -C \int |x|^s \de \mu(x) - D, \quad \forall \mu \in \Prob(\setR) \, .
	\end{align}
	\item\label{i:A2}
	$F:(\Prob(\setR),\bW_q) \to (-\infty,+\infty]$ is proper and lsc, for some $q<2$. 
\end{enumerate}

\begin{example}
	\label{ex:potential_energies}
	Fix $f:\R \to \R \cup \{+\infty\}$. Then the functional 
	\begin{align}
	\label{eq:potential_energies}
		F(\mu) = \int f(x) \de \mu(x), \quad [f(x)]_{\_}\leq C|x|^s + D
	\end{align}
	satisfies \eqref{eq:lowercontrolF-pmoments}. If $f$ is lower semicontinuous, also the assumption \eqref{i:A2} is satisfied \cite[Lemma~5.1.7]{AmbrosioGigliSavare08}. A simple example is given by the linear drift $f(x) = \pm x$.
\end{example}

A consequence of the above assumptions is the well-posedness of the variational problems in \eqref{eq:JKOtauN}, \eqref{eq:JKOtau}. 
Fix $N \in \N$ and two functionals $ F_N$, $F$ as above, and define the discrete JKO functionals as
\begin{align}
	\label{eq:def_JtauN}
	\cJ_{\tau,N}: \Prob(\setR) \to [0,+\infty], \quad \cJ_{\tau,N} (\mu) := F_N(\mu) +  \frac1{2\tau}\bNtwoW(\mu, \mu_N^{(0)})^2 \;,
\end{align}
whereas the continuous JKO functional are denoted by
\begin{align}
	\label{eq:def_Jtau}
	\cJ_\tau: \Prob(\setR) \to [0,+\infty], \quad  \cJ_\tau(\mu):= F(\mu) + \frac1{2\tau} \btwoW(\mu,\mu^{(0)})^2 \; .
\end{align}
Recall the definition of $\theta$ asociated to a nonlinear mobility $m$ as in Definition~\ref{def:theta}.
\begin{proposition}[Well-posedness]	\label{prop:well-posedness}
	Fix $N \in \N$ and assume that $\theta \in \emph{L}^\infty(0,M)$. Let $F$, $F_N$ be two functionals satisfying \eqref{i:A1},\eqref{i:A2}.
	\begin{enumerate}[(i)]
		\item 
		\label{i:prop:well-posedness_1}
		If the initial measure $\mu_N^{(0)}$ belongs to $\Prob_2(\setR)$, then the minimisation problems \eqref{eq:JKOtauN} admit solutions, for every $ n \in \setN$. 
		\item 
		\label{i:prop:well-posedness_2}
		The property $(i)$ holds as well if $\bNtwoW$ is replaced with $\bNhtwoW$ in \eqref{eq:JKOtauN}.
		\item 
		\label{i:prop:well-posedness_3}
		If $\mu^{(0)} \in \Prob_2(\setR)$, then the continuous minimisation problems \eqref{eq:JKOtau} admit solutions, for every $n \in \setN$.  
	\end{enumerate}

\end{proposition}

We first need the following preliminary result, concerning the relation between the weighted distance $\bpW$ and the classical Wasserstein $\bW_p$.

\begin{lemma}[Comparison with Wasserstein distances] \label{lemma:bounds_distances}
	Assume $\theta \in \emph{L}^{\infty}(0,M)$. Then it holds
	\begin{align} \label{eq:lowbound_distances_cont}
	\bpW(\mu_0,\mu_1)^p \geq \frac{1}{\| \theta \|_{\infty}^{p-1}} \bW_p(\mu_0,\mu_1)^p, \quad \forall \mu_0, \mu_1 \in \Prob(\setR).
	\end{align}
Similarly,  for every $N \in \N$, we have that 
	\begin{align}	\label{eq:lowbound_distances_discr}
	\bNhW(\mu_0,\mu_1)^p \geq \frac{1}{\| \theta \|_{\infty}^{p-1}} \bW_p(\mu_0, \mu_1)^p, \quad \forall \mu_0, \mu_1 \in \E^N.
	\end{align}
	\end{lemma}

\begin{proof}
	The lower bound \ref{eq:lowbound_distances_cont} is trivial, whereas \ref{eq:lowbound_distances_discr} follows from \eqref{eq:Wass_pseudo}. 
\end{proof}

\begin{proof}[Proof of Proposition~\ref{prop:well-posedness}]
	\eqref{i:prop:well-posedness_1}, \eqref{i:prop:well-posedness_2}: \
	We prove the claim for \eqref{eq:JKOtauN} and $n=1$, the proof for a general $n$ follows by iteration. In the discrete setting, we work with the cone $\cK_N$ which is a subset of the finite dimensional euclidean space $\setR^{N+1}$.  
	
	The lower semicontinuity with respect to $\bW_q$ of $\cJ_{\tau,N}$ follows from the one of $F$ and the continuity of $d_{2,m}^N$, together with the closedness of $\cK_N$.
	Thus, to prove the existence of minimisers, it is enough to show that the family $\seq{\cJ_{\tau,N}}_N$ is suitably coercive in $(\cP_q(\R),\bW_q)$. From \eqref{i:A1} and Lemma \ref{lemma:bounds_distances}, by means of a triangle inequality and the monotonicity of the $\ell$-norms on $\cK_N$ with $s<2$, we obtain for every $\mu \in \Prob_2(\setR)$
	\begin{align}	
		\begin{aligned}
			\label{eq:equi-coercivity-JKO}
		\cJ_{\tau,N}(\mu) 
			&\geq \frac{1}{\|\theta\|_\infty} \bW_2 (\mu, \mu_N^{(0)} )^2 - C \bE_{\mu}|x|^s - D \\
			&\geq \frac1{2\|\theta\|_\infty} \bE_\mu|x|^2 - \frac1{\|\theta\|_\infty}\bE_{\mu_N^{(0)}} |x|^2 - C \bE_{\mu}|x|^s - D \\
			&\geq \frac1{2\|\theta\|_\infty} \bE_\mu|x|^2 - C \Big( \bE_\mu|x|^2 \Big)^{s/2}  - \frac1{\|\theta\|_\infty}\bE_{\mu_N^{(0)}} |x|^2 -D  \\
			&\geq c \bE_\mu|x|^2 - \frac1{\|\theta\|_\infty}\bE_{\mu_N^{(0)}} |x|^2 -D' \, ,
		\end{aligned}
	\end{align}
	where $c$, $D' \in \setR^+$ only depends on $\theta$, $C$, and $D$.
	Under our assumptions, this shows that the sublevels of $\cJ_{\tau,N}$ are bounded in $\bW_2$. In particular, for every $t \in \setR$, we have
	\begin{align*}
		\Big\{
			\mu \in \Prob(\setR) \suchthat \cJ_{\tau,N}(\mu)\leq t
		\Big\}
			\subset
		\Big\{
			\mu = \E^N(x) \suchthat 
				\| x \|_{2,N} \leq C
					\; \text{and} \;
				N(\Delta x)_i \geq 1/M
		\Big\} \, ,
	\end{align*}
	for some $C \in \setR_+$ depending on $t$. This easily implies coercivity of $\cJ_{\tau,N}$ and thus the existence of a minimiser. \eqref{i:prop:well-posedness_2} is trivial.
	
	\medskip
	\noindent
	\eqref{i:prop:well-posedness_3}: \
	By similar computations as in \eqref{eq:equi-coercivity-JKO}, one can show the $\bW_q$-coercivity of $\cJ_\tau$ using \eqref{eq:lowbound_distances_cont}. The lower semicontinuity with respect to $\bW_q$ follows from \eqref{i:A2}.
\end{proof}
\begin{remark}[Uniqueness]
	In general, no uniqueness is guaranteed. Whenever $F$, $F_N$ satisfy a suitable strict convexity assumption, then the scheme is uniquely defined. Nevertheless, we choose to work without such assumption, due to some interesting, non-strictly convex examples of interest (e.g. traffic flow evolutions). 
\end{remark}
\begin{proposition}[Equicoercivity]
\label{prop:equicoercivity}
Let $\theta \in \emph{L}^\infty(0,M)$. 
		Let $\big( \mu_N^{(0)}\big)_N$ be a sequence of initial measures with equibounded second moments, i.e. 
		\begin{align}
		\label{eq:unif_bounds_2ndmoms}
		\sup_{N \in \setN} \bE_{\mu_N^{(0)}} |x|^2 < \infty.
		\end{align}
		Assume that the family $\seq{F_N}_N$ satisfies the assumption in \eqref{i:A1} with some uniform (in $N \in \N$) constants $C,D \in \R_+$. 		
		 Then the functionals $\seq{ \cJ_{\tau,N} }_N$ are equicoercive with respect to the $\bW_q$ distance, for every $q \in [1,2)$: thus, for every $\lambda >0$, the set $\{ \mu \in \cP(\R) \, : \, \cJ_{\tau,N}(\mu) \leq \lambda \}$ is compact in $(\cP_q(\R), \bW_q)$.
\end{proposition}
\begin{proof}
	It follows directly from \eqref{eq:equi-coercivity-JKO} and from the fact that the balls of $\bW_2$ are compact in $\bW_q$, for every $q <2$, see e.g. \cite[Lemma 5.1.7]{AmbrosioGigliSavare08}.
\end{proof}
\begin{remark}[Alternative coercivity]
	The conclusion of Proposition~\ref{prop:equicoercivity} remains valid if we replaced the assumption in \eqref{i:A1} with $\seq{F_N}_N$ being equicoercive w.r.t. some $\bW_s$, $s \in [1, +\infty)$ (and it holds with $q=s$).
\end{remark}
\begin{remark}[Lack of $\bW_2$-coercivity]
	\BBB [MANU-SIMO: controllate se vi piace e cancellate questo] \BLK In general, without any additional assumption of $\seq{ F_N }_N$ (as discussed in the previous remark), we are not able to show the $\bW_2$-convergence of the JKO schemes. At least for certain energies, e.g.  the one in \eqref{eq:potential_energies} with a Lipschitz $f$, one may expect to be able to obtain such stronger convergence. Unfortunately, new ideas seem to be required to obtain coercivity in this topology of the energies involved in the JKO schemes, including a more detailed analysis of the discrete geodesics on $\cK_N$. We leave this to future analysis. 
\end{remark}
\section{Statements of the main results}
\label{sec:main}
We are ready to present the main contributions of this work. The first one is the convergence, in the sense of $\Gamma$-convergence as $N \to \infty$, of the discrete trasport distances $\bW_{p,m}^N$, $\bNhW$ to the continuous one $\bpW$, with respect to a broad class of weak topologies on $\Prob(\setR)$. 

\begin{theorem}[$\Gamma$-convergence of $\bNW$ to $\bpW$] \label{thm:Gamma-conv}
Fix $1<p<\infty$, $M \in \R_+$, and let $\phi_{p,m}$ be as in Definition~\ref{def:action_density_functions}. Let $\theta$ be as in Definition~\ref{def:theta} and assume that $\theta \in \emph{L}^\infty(0,M)$. Then we have $\Gamma$-convergence of $\bNW$ towards $\bpW$ with respect to the vague and the $\bW_q$-topology, for any $1<q<+\infty$. More precisely, we have the following bounds:
	\begin{enumerate}
		\item[(\emph{LI})] For every $\mu_0^N,\mu_1^N \to \mu_0,\mu_1$ vaguely in $\cP(\R)$ as $N \rightarrow +\infty$, we have that
		\begin{align*}
			\liminf_{N \rightarrow +\infty} \bNW(\mu_0^N, \mu_1^N) \geq \bpW(\mu_0,\mu_1).
		\end{align*}
		\item[(\emph{LS})] For any given $1<q<+\infty$, for every $\mu_0,\mu_1 \in \Prob_q(\setR)$,  there exist  $\tilde{\mu}_0^N, \tilde{\mu}_1^N \rightarrow \mu_0, \mu_1$ in $(\cP_q(\R),\bW_q)$ as $N \rightarrow +\infty$ and such that 
		\begin{align*}
			\limsup_{N \rightarrow +\infty} \bNW(\tilde{\mu}_0^N, \tilde{\mu}_1^N) \leq \bpW(\mu_0,\mu_1).
		\end{align*}
	\end{enumerate}
The very same conclusions hold if we replace $\bNW$ with $\bNhW$. 
\end{theorem}

The second main result concerns the corresponding gradient-flow structures: we have that, at every time-scale $\tau>0$, the sequences $\seq{\mu_{\tau,N}^{(n)}}_n$ defined via \eqref{eq:JKOtauN} converge as $N \to \infty$ to the corresponding continuous (in space) steps of \eqref{eq:JKOtau}. 
\begin{theorem}[JKO-convergence in $\bW_q$, for $q <2$]
	\label{thm:JKO_q}
	Consider the same assumptions of Theorem~\ref{thm:Gamma-conv}. Fix $\tau \in \R_+$, $q \in [1,2)$, and let $\seq{F_N}_N$, $F$ be  functionals satisfying \eqref{i:A1},\eqref{i:A2}, and 
		\begin{align}
		\label{eq:cont_F_N_F}
		\lim_{N \to \infty} F_N(\mu_N) = F(\mu)
		\, , \quad 
		\text{for every} \ \  
		\mu_N \to \mu 
		\ \  \text{in} \ (\cP_q(\R),\bW_q) \, .
	\end{align}
	Let then $\seq{ \mu_N^{(0)}}_N$ be a sequence satisfying \eqref{eq:unif_bounds_2ndmoms}, and suppose that $\mu_N^{(0)} \to \mu^{(0)}$ in $(\cP_q(\R),\bW_q)$.
	 Then we have that
	\begin{enumerate}[(i)]
		\item 
		\label{i:thm:JKO_q_1}
		$\cJ_{\tau,N}$ $\Gamma$-converge to $\cJ_\tau$ with respect to the $\bW_q$-topology, as $N \to \infty$.
		\item 
		\label{i:thm:JKO_q_2}
		Up to subsequences, $\mu_{\tau,N}^{(n)} \to \mu_\tau^{(n)}$ in $(\cP_q(\R),\bW_q)$ as $N \to \infty$, for all $n \in \setN$.
	\end{enumerate}
\end{theorem} 

\begin{remark}[The case $F_N \equiv F$]
	The easiest example of a family satisfying \eqref{eq:cont_F_N_F} is when $F_N:= F$, for some $F$ which is $\bW_q$-continuous (thus \eqref{i:A2} holds as well).
\end{remark}

\section{Approximation of finite-energy solutions of $\CE$}
\label{sec:approx}
In this chapter we provide some approximation results (Proposition~\ref{prop:approximation}) for solutions to the continuity equation with bounded action. These techniques allow us to obtain "good" recovery sequences for $\bpW$, which turn out to be crucial in the proof of Theorem~\ref{thm:Gamma-conv} and Theorem~\ref{thm:JKO_q}. We start recalling some basic notions.

\begin{definition}[Pseudo-inverse]
	Fix $\mu \in \Prob(\setR)$. The \textit{pseudo-inverse} or \textit{quantile function} $X_\mu:(0,1) \to \setR$ of $\mu$ is the non-decreasing, right-continuous map given by
	\begin{align*}
			X_{\mu}(z) 
		= \inf \left \{ a \in \setR \suchthat \mu (-\infty, a) > z \right\} \, .
	\end{align*}
\end{definition}

\begin{remark}[Law of the pseudo-inverse]
	It is easy to check that, by construction, the law of $X_\mu$ coincides with $\mu$ itself, i.e. $(X_\mu)_{\#} \Leb^1 = \mu$. In fact, it is the unique right continuous monotone increasing map from $(0,1)$ to $\R$ with such property.
\end{remark}

\begin{remark}[Wasserstein distance on the real line]
	\label{rem:Wass_pseudo}
A remarkable property of the one-dimensional optimal transport problem is the possibility of computing the $p$-Wasserstein distance between two probability measures via the $\tL^p$-distance of the corresponding pseudo-inverses. Precisely, for $p\geq 1$, and every $\mu,\nu \in \Prob_p(\setR)$, it holds
\begin{align}
	\label{eq:Wass_pseudo}
\bW_p(\mu,\nu) = \| X_\mu - X_\nu \|_{\tL^p(0,1)} \, .
\end{align}
\end{remark}

\begin{remark}[Monotonocity and continuity of the pseudo-inverse]
	\label{rem:pseudo_prop}
	Whenever $\mu$ does not charge single points, $X_\mu$ is an increasing map (and in particular, injective). In general, even for absolute continuous measures (with respect to $\Leb^1$) with smooth densities, the corresponding pseudo-inverse might fail to be continuous. If the support of $\mu$ is connected, then $X_\mu$ is continuous.
\end{remark}

\subsection{Reduction to compact support}
In this section, we discuss a procedure to reduce measures to compactly supported ones, and how this effects the value of the associated $\phi_{p,m}$-energies.
\begin{definition}[Compactification of a measure]	\label{def:rescaling}
	Fix $\mu \in \Prob(\setR)$ an atomless probability measure, $\nu \in \cM(\setR)$. For every $\eps>0$, we define the measure $\cC_\eps (\nu | \mu) \in \Prob(\setR)$ as 
	\begin{align}
			\label{eq:def_rescaling}
		\cC_\eps (\nu|\mu):= \frac{1}{1-2\eps} \nu \res I_\eps[\mu] \, , 
			\quad
		I_\eps[\mu]:= 
			\Big[
				X_\mu(\eps) , X_\mu(1-\eps)
			\Big] \, .
	\end{align}
	We use the shorthand notation $\cC_\eps \mu := \cC_\eps(\mu|\mu) \in \Prob(\setR)$.
\end{definition}

The next lemma provides a control of the error in applying the operator $\cC_\eps$.
\begin{lemma}[Compactification error]	\label{lemma:rescaling_error}
	Fix $\mu \in \Prob_p(\setR)$ an atomless probability measure, $\eps>0$. Then for every $q\leq p$,
\begin{align}	\label{eq:rescaling_bound_1}
	\emph{diam}(I_\eps[\mu])^{p-q} \bW_q(\mu,\cC_\eps \mu)^q 
	\leq
	C\int_{I_\eps[\mu]^c} |x - x_\mu|^p \de \mu(x) \, ,
\end{align}
for some constant $C>0$ depending only on $p$, where $x_\mu:= X_\mu\big( \frac12 \big)$ is the median of $\mu$. If $q=p$, one can take $C=1$.
\end{lemma}

\begin{proof}
	Without loss of generality, we can assume $x_\mu=0$, otherwise we change variables $x \mapsto x- x_\mu$. By density arguments, we can further assume that $\mu$ is absolutely continuous with respect to $\Leb^1$.
	We split the proof in two steps.
	
	\medskip
	\noindent
	\textit{Step 1: the case $q=p$.} \
	Pick any admissible map $\tilde T$ between $\mu \res I_\eps[\mu]^c$  and $2\eps \cC_\eps \mu$ (i.e. $\tilde T_{\#} ( \mu \res I_\eps[\mu]^c ) = 2\eps \cC_\eps \mu$). Moreover, thanks to $x_\mu=0$, we can also assume that $\tilde T(x) x \geq 0$, for every $x \in \setR$.
	Then, we define the map
	\begin{align*}
		T:\setR \to I_\eps[\mu] 
			\, , \quad
		T(x) =
		\begin{cases}
			x ,
				&\text{if }x \in I_\eps[\mu] \, , \\
			\tilde T(x) , 
				&\text{otherwise} \, .
		\end{cases}
	\end{align*}
Note that $T$ is an admissible map between $\mu$ and $\cC_\eps \mu$. Indeed for every measurable map $\vphi:\setR \to \setR$, we have that
\begin{align*}
	\int \vphi(T(x)) \de \mu(x) 
		&= \int_{I_\eps[\mu]} \vphi \de \mu + \int_{I_\eps[\mu]^c} \vphi(\tilde T(x)) \de \mu(x) \\
		&= (1-2\eps) \int \vphi \de \cC_\eps \mu + 2\eps \int \vphi \de \cC_\eps \mu = \int \vphi \de \cC_\eps \mu \, .
\end{align*}
Computing the transport cost with respect to $|\cdot|^p$, we obtain that
\begin{align*}
	W_p^p(\mu,\cC_\eps \mu) \leq \int |x-T(x)|^p \de \mu(x) = \int_{I_\eps[\mu]^c} |x - \tilde T(x) |^p \de \mu(x) \leq \int_{I_\eps[\mu]^c} |x|^p \de \mu(x) \, ,
\end{align*}
where at last we used $\tilde T(x)x\geq 0$, which concludes the proof of \eqref{eq:rescaling_bound_1}.

\medskip
\noindent
\textit{Step 2: $q<p$}. \
Fix $\mu \in \Prob_p(\setR)$ and $q<p$. Then applying the first step of the proof and using that $\mu \in \Prob_q(\setR)$ we deduce 
\begin{align}	\label{eq:q-bound}
	\bW_q(\mu,\cC_\eps \mu)^q
	\leq
	\int_{I_\eps[\mu]^c} |x|^q \de \mu(x) \, .
\end{align}
Consider any transport map $S: I_\eps[\mu]^c \to I_\eps[\mu]^c$ satisfying
\begin{align}
	\label{eq:law_S}
	S_{\#} (\mu \res I_\eps[\mu]^c) = \mu \res I_\eps[\mu]^c
		\quad \text{and} \quad 
	S(x)x \leq 0
		\, , \quad 
			x \in I_\eps[\mu]^c \, .
\end{align}
One can for example consider any map that sends the $\eps$ mass of $\mu$ which lives in $I_\eps[\mu]^c \cap \{x \leq 0 \}$ into $I_\eps[\mu]^c \cap \{x \geq 0 \}$ and viceversa. In particular, we observe that $|x| \leq |x - S(x)|$, for $x \in I_\eps[\mu]^c$. Together with \eqref{eq:q-bound}, this yields 
\begin{align*}
	\diam(I_\eps[\mu])^{p-q} \bW_q(\mu,\cC_\eps \mu)^q
		&\leq 
			\diam(I_\eps[\mu])^{p-q} 
				\int_{I_\eps[\mu]^c} |x - S(x)|^q \de \mu(x)	
					\\
		&=
			\int_{I_\eps[\mu]^c}	
				\bigg(
				 	\frac{\diam(I_\eps[\mu])}{|x - S(x)|}
				 \bigg)^{p-q} 
			 |x - S(x)|^p \de \mu(x) 	
			 	\\
		&\leq 
				\int_{I_\eps[\mu]^c} 
					|x - S(x)|^p \de \mu(x) \, ,
\end{align*}
where at last we used that for $x \in I_\eps[\mu]^c$ we have $|x-S(x)| \geq \diam(I_\eps[\mu])$. By \eqref{eq:law_S}
\begin{align*}
	\int_{I_\eps[\mu]^c} 
		|x - S(x)|^p \de \mu(x)
	&\leq
		2^{p-1} \int_{I_\eps[\mu]^c}
			\big( |x|^p + |S(x) |^p \big) \de \mu(x)
			=
		2^p \int_{I_\eps[\mu]^c}
			|x |^p \de \mu(x) \, , 
\end{align*}
we conclude the proof.
\end{proof}

The operator $\cC_\eps$ is constructed in such a way to preserve solutions to $\CE$.
\begin{lemma}[Compactification of solutions to $\CE$]	\label{lemma:rescaling_CE}
	Let $(\mu_t,\nu_t) \in \CE$ with $C^1$ (space-time) densities $(\rho_t,j_t)_t$ with respect to $\Leb^1$. Assume also that  $\rho_t>0$ for every $t \in [0,1]$. For every $\eps>0$, consider the operator $\cC_\eps$ as defined in Definition \ref{def:rescaling}. 
	Then 
	$$
		\big( \cC_\eps \mu_t, \cC_\eps (\nu_t|\mu_t) \big) \in \CE \, .
	$$ 
\end{lemma}

\begin{proof}
	Fix any $\vphi \in C_c(\setR)$. For the sake of simplicity, set $a_t:=X_{\mu_t}(\eps)$, $b_t:= X_{\mu_t}(1-\eps)$. Firstly, for $t \in (0,1)$, set $F_\mu(x):= \mu(-\infty,x)$ and note that by definition of quantile function one has that 
	\begin{align*}
		F_{\mu_t} \big(  X_{\mu_t} (z) \big) = z
			\,  ,\quad 
		\forall z \in (0,1) \, .
	\end{align*}
	Using that $(\mu_t, \nu_t) \in \bCE$, it is readily verified that $\frac{\de}{\de t} F_{\mu_t}(x) = - j_t(x)$, for every $t \in [0,1]$, $x \in \R$. Note that $(t,x) \mapsto F_{\mu_t}(x)$ is $C^1$ in space and time. By applying the implicit function theorem, we infer that $t \mapsto X_{\mu_t}(z)$ is differentiable for every $z \in (0,1)$ and it holds
	\begin{align}
		\label{eq:equation_pseudoinverse_time}
		- j_t \big( X_{\mu_t}(z) \big)
			+ \rho_t\big(  X_{\mu_t}(z) \big) 
		\frac{\de}{\de t}X_{\mu_t}(z)
		= 0 \, .
	\end{align}
	Hence, $a_t$, $b_t$ both solves the differential equation $ \rho_t(c_t) \dot c_t= j_t(c_t)$.
	 Therefore, we have that
	\begin{align*}
		(1-2\eps)\frac{\de}{\de t} \int \vphi \de \cC_\eps \mu_t 
			&= \frac{\de}{\de t} \int_{I_\eps[\mu_t]} \vphi \de \mu_t \\
		&= \int_{I_\eps[\mu_t]} \vphi  \dot{\rho}_t \de x + \dot b_t \vphi(b_t) \rho_t(b_t) - \dot a_t \vphi(a_t) \rho_t(a_t) 
	\\
			&=-\int_{I_\eps[\mu_t]} \vphi \nabla \cdot j_t \de x + \vphi j_t\big|_{\partial I_\eps[\mu_t]} 
				\\
				&= \int_{I_\eps[\mu_t]} \nabla \vphi \cdot j_t = (1-2\eps) \int \nabla \vphi \de \cC_\eps (\nu_t|\mu_t) \, . \qedhere
	\end{align*}
\end{proof}

In view of the error estimate of Lemma~\ref{lemma:rescaling_error}, we also prove the following result.
\begin{lemma}[Uniform bounds and stability properties]
		\label{lemma:stability_Ieps}
	Fix $\eps>0$, $\seq{ \mu_n}_n \subset \Prob(\setR)$ a sequence of atomless measures. 
	\begin{enumerate}[(i)]
		\item \label{i:lemma:stability_Ieps:diam_tight} 
		If $\seq{ \mu_n}_n $ is tight in $\Prob(\setR)$, then for $s \in (0,1)$ 
		\begin{align}
				\label{eq:diam_tightness}
			\sup_{n \in \setN}
				\emph{\diam}(I_\eps[\mu_n]) < \infty
			\quad \text{and} \quad
				\sup_{n \in \setN}
					 \big| X_{\mu_n}(s) \big| <\infty \, .
		\end{align}
		\item 
		\label{i:lemma:stability_Ieps_2} 
		For every $\mu,\nu \in \cP_p(\setR)$ atomless measures, it holds
		\begin{align*}
			\left|					
				\int_{I_\eps[\mu]^c} |x|^p \de \mu(x) 				
					-				
				\int_{I_\eps[\nu]^c} |x|^p \de \nu(x) 				
			\right|^{\frac1p}
				\leq 
			\Big\|
				X_\mu - X_\nu
			\Big\|_{\emph{L}^p( 
							[\eps,1-\eps]^c
						)} .
		\end{align*}
		In particular, if $\mu_n \to \mu$ in $\bW_p$, then 
		\begin{align}
				\label{eq:tails_stability}
			\lim_{n \to \infty} 
				\int_{I_\eps[\mu_n]^c} |x|^p \de \mu_n(x) 
					=
				\int_{I_\eps[\mu]^c} |x|^p \de\mu(x) \, . 
		\end{align}
	\end{enumerate}
\end{lemma}

\begin{proof} \, 
	\eqref{i:lemma:stability_Ieps:diam_tight}: \
	We proceed by contradiction: assume that (up to subsequence) 
	\begin{align}
		\label{eq:contradd}
			\lim_{n \to \infty} 
				\diam(I_\eps[\mu_n]) = +\infty \, .
	\end{align}
	By tightness, we know that there exists a compact set $K_\eps \Subset \setR$ such that $\sup_n  \mu_n(K_\eps^c) < \eps$. On the other hand, by definition of $I_\eps$, we know that $I_\eps[ \mu_n]^c$ is a disjoint union of unbounded intervals $I_n^-$ and $I_n^+$, each of $ \mu_n$-mass equals to $\eps$. From \eqref{eq:contradd}, we deduce that, for $n$ sufficiently big, either $I_n^-$ or $I_n^+$ must be subsets of $K_\eps^c$, which gives us the contradiction.
	A similar argument shwos the uniform bound on the quantiles.
	
	\medskip
	\noindent
	\eqref{i:lemma:stability_Ieps_2}: \
	Fix $\mu \in \cP_p(\setR)$. By definition of pseudo-inverse we have
	\begin{align*}
		\int_{I_\eps[ \mu]^c} |x|^p \de  \mu(x)
		=
			\int_0^\eps 
			\Big|
				X_{ \mu}(z)
			\Big|^p
				\de z
		+
			\int_{1-\eps}^1 
			\Big|
				X_{ \mu}(z)
			\Big|^p
		\de z \, .
	\end{align*}
	The assumption on the $\bW_p$ convergence on $\mu_n$ is equivalent to the fact that $X_{ \mu_n} \to X_{\mu}$ in $\tL^p(0,1)$ (see Remark \ref{rem:Wass_pseudo}), and thus $(ii)$ follows. 
\end{proof}

\subsection{The approximation procedure}

We introduce the following

\begin{definition}[Regular curves]
\label{def:regular_measures}
Let  $(\mu_t,\nu_t) \in \CE$ be a solution of the continuity equation in $\setR^d$
.  We say that $(\mu_t,\nu_t)$ is a \emph{regular curve} if the following properties hold 
\begin{enumerate}
\item[i)] for all $t \in [0,1]$, $\mu_t$ and $\nu_t$ have compact and connected support, $\mu_t, \nu_t \ll \Leb^1$ and their respective densities densities $\rho_t$, $ j_t$ are $C^\infty$ in their support
\item[ii)] there exist a compact set $K$ and two constants $0 < l \leq \tilde M < M$ such that 
\[ \supp(\rho_t), \, \supp( j_t) \subset K  \quad \text{and} \quad  l  < \rho_t(x) < \tilde{M} \quad \text{for all  $x \in \supp(\rho_t)$ and $t \in [0,1]$.} \]
\end{enumerate}
\end{definition}

We are finally ready to state and prove the main result of this section.

\begin{proposition}[Approximation with regular curves]
		\label{prop:approximation}
	Let $p \in (1,	\infty)$, $M \in \R_+$, and $\phi_{p,m}$ be as in Definition~\ref{def:action_density_functions}. Let
	$(\mu_t,\nu_t) \in \CE$ satisfying
	\begin{align}	\label{eq:assumption_approximation_prop}
		E:= \int_0^1 \Phi_{p,m}(\mu_t, \nu_t) \de t < \infty \, , 
			\quad 
		\sup_{t \in [0,1]} \| \rho_t \|_{\emph{L}^\infty(\R)} =: \bar M < M \, 
 \, .
	\end{align}
Then, for every $\eta >0$, there exists $(\tilde \mu_t, \tilde \nu_t) \in \CE$
such that
\begin{enumerate}[(i)]
	\setlength{\itemsep}{.2\baselineskip}
	\item 
	\label{i:prop:approximation_1}
	$	
	\displaystyle
		\int_0^1 \Phi_{p,m}(\tilde \mu_t, \tilde \nu_t) \de t \leq \int_0^1 \Phi_{p,m}(\mu_t, \nu_t) \de t + \eta \, .
	$
	\item 
	\label{i:prop:approximation_2}
	$(\tilde \mu_t, \tilde \nu_t)$ are regular curves in the sense of Definition \ref{def:regular_measures}, for some compact set $K \Subset \setR$ and $l>0$ depending only on $\eta$ and $(\mu_t)_{t \in[0,1]}$,  while $\tilde M$ only depends on $\bar M$.
	\item 
	\label{i:prop:approximation_3}
	For $q \leq p$ and $t \in [0,1]$, we have that
	\begin{align*}
		 \max
		 \left\{ 
		 	\Big( \emph{diam}(\supp(\tilde \mu_t))^{p-q} \vee 1 \Big)
			\bW_q(\mu_t, \tilde \mu_t)^q 
		 \, ,  
			\int_{\supp(\tilde \mu_t)^c} (|x|^p + 1) \de \mu_t(x)
		\right\}			
			\leq \eta \, .
	\end{align*}

\end{enumerate}
\end{proposition}

\begin{proof}
	We split the proof into several steps. 
	
	\medskip
	\noindent
	\textit{Step 1: Space-time regularisation}. \ 
	We extend $\mu_t,\nu_t$ as $\mu_t = \mu_0$ for $t <0$, $\mu_t = \mu_1$ for $t >1$, and $\nu_t = 0$ for $t \in (-\infty,0) \cup (1,+\infty)$. Given two smooth functions $f,g:\R \to \R_+$ such that
	\begin{align}
	\label{eq:prop_f_g}
		\int f(x) \de x = 1
	\, ,\quad 
		\int|x|^p f(x) \de x < \infty 
	\, , \quad 
		f>0 \ \text{on} \ \R  \, , 
	\end{align}
	and the same for $g$. We define for $\tau >0$ the rescaled kernel on $\R \times \R$ by setting $h_\tau(t,x):= \tau^{-2} f(\tau^{-1}t) g(\tau^{-1}x)$, for $(t,x) \in \R \times \R$. Note that $h_\tau$ has integral $1$ and full support, for every $\tau>0$. We then consider $(\bfmu, \bfnu)$ as in \eqref{eq:bnu} and define 
	\begin{align}	\label{eq:def_regular_rho}
		\hat \bfmu:= h_\tau * \bfmu 
		\, ,  \quad 
		\hat \bfnu:= h_\tau * \bfnu \, , 
	\end{align}
	where the convolution has to be intended as in Section~\ref{subsec:auxiliary}\eqref{i:convolution}. In particular,
	\begin{align}	\label{eq:energy_decrease_tau}
		\int_0^1  \Phi_{p,m}(\hat \mu_t,\hat \nu_t) \de t \leq \int_0^1 \Phi_{p,m}(\mu_t,\nu_t) \de t \, .
	\end{align}
	Moreover, $\hat \bfmu$, $\hat \bfnu$ have densities $\hat \rho_\cdot(\cdot), \hat j_\cdot(\cdot) \in C^\infty([0,1]\times \setR)$. Thanks to the positivity of $h_\tau$ and $\int h_\tau = 1$, we have that 
$
		\hat \rho_t(x)\in (0,\bar M)
$,
	for every $x \in \setR$, $t \in [0,1]$. 
	
	\medskip
	\noindent
	\textit{Step 2: Reduction to compact supports}. \
	Pick $\bar \eps = \eps(\bar M) \in (0,1)$ so that 
	\begin{align}	\label{eq:M_tau<M}
		\tilde M:= (1-2\bar \eps)^{-1} \bar M < M \, .
	\end{align}
	
	For $\eps \in (0,\bar \eps)$ to be suitably chosen and $t \in [0,1]$, we define 
	\begin{gather*}
		\tilde \mu_t:= \cC_\eps \hat \mu_t 
		\, , \quad
		\tilde \nu_t:=  \cC_\eps(\hat \nu_t | \hat \mu_t) \, ,
	\end{gather*}
	where $\cC_\eps$ is defined in Definition \ref{def:rescaling}. We denote by $\tilde \rho_t$, $\tilde j_t$ the corresponding densities.
	\medskip
	\noindent
	\textit{Step 3: Energy estimates}. \
	Thanks to Lemma \ref{lemma:rescaling_CE} and the fact that $(\hat \mu_t,\hat \nu_t) \in \CE$, we know that $(\tilde \mu_t, \tilde \nu_t) \in \CE$.
	By the convexity and homogeneity of $\phi_{p,m}$, we have that
	\begin{align*}
		\phi_{p,m}(\tilde \rho_t, \tilde j_t) = \frac{1}{(1-2\eps)^p} \phi_{p,m} 
		\Big(
		\frac{1}{1-2\eps} \hat \rho_t, \hat j_t
		\Big) 
		\leq 
		\frac{1+\eps \bar C}{(1-2\eps)^p}		\phi_{p,m}(\hat \rho_t, \hat j_t) \, , 
	\end{align*}
	for some $\bar C= \bar C(\bar M,\phi_{p,m})<\infty$ depending only on $\bar M$ and $\phi_{p,m}$, where at last we used Lemma \ref{lemma:rescaling}, thanks to \eqref{eq:M_tau<M}. In particular, by  \eqref{eq:assumption_approximation_prop}, \eqref{eq:energy_decrease_tau}, and the latter bound,
	\begin{align}	\label{eq:proof_energy_esti_tilde}
		\int_0^1 \Phi_{p,m}(\tilde \mu_t, \tilde \nu_t) \de t 
		\leq
		\frac{1+\eps \bar C}{(1-2\eps)^p}	\int_0^1 \Phi_{p,m}(\mu_t, \nu_t) \de t  
			\leq \int_0^1 \Phi_{p,m}(\mu_t, \nu_t) \de t + \bar C E \eps 
		\, .
	\end{align}
	
	The upper bound \eqref{i:prop:approximation_1} then follows directly by choosing $\eps>0$ small enough. 
	
	\medskip
	\noindent
	\textit{Step 4: Regularity of the approximation}. \
	By construction and \eqref{eq:M_tau<M}, we have that
	\begin{align}	\label{eq:proof_regularity_tilde}
		\begin{gathered}
			\big\{ t \mapsto \tilde \rho_t, \, \tilde j_t \in C^\infty(\supp(\tilde \mu_t)) \big\} \in C^\infty([0,1])
			\, , 
		\quad  
				\supp \tilde \mu_t = I_\eps[\hat \mu_t]  \subset \setR 
			\, , \\
			0 < l_\tau \leq \tilde \rho_s(x) \leq \tilde M <M \, , \quad \forall x \in \supp(\tilde \mu_s) \, ,
		\end{gathered}
	\end{align}
	for every $s \in [0,1]$, for some $l_\tau>0$ depending only on $\tau$, $f$, and $(\mu_t)_{t \in[0,1]}$ (but not on $\eps$, since $(1-2\eps)^{-1} >1$). Recall the definition of $I_\eps[\cdot]$  as given in \eqref{eq:def_rescaling}.
	In particular, \eqref{i:prop:approximation_2} follows with a compact set $K=K_{\tau,\eps} \subset \setR$ only depending on $\eps$, $\tau$, and $(\mu_t)_{t \in[0,1]}$. 
	
	\medskip
	\noindent
	\textit{Step 5: Estimate of the distance.} \
	We are left with the proof of \eqref{i:prop:approximation_3}. By triangle inequality, we have that
		$
		\bW_q(\mu_t, \tilde \mu_t) 
		\leq 
			\bW_q(\mu_t, \hat \mu_t) +
				\bW_q(\hat \mu_t, \tilde \mu_t) \, .
	$
	Set $f_\tau := \tau^{-1} f(\tau^{-1}\cdot)$. By Remark \ref{rem:time_reg_finite_energy} with $q \leq p$ and \eqref{eq:prop_f_g}, for every $s,t \in [0,1]$ we have that 
	\begin{align*}
		\bW_q(\mu_t, f_\tau * \mu_s) 
			\leq
				 \bW_q(\mu_t, \mu_s) +
				 	\bW_q(\mu_s, f_\tau * \mu_s)
			\leq
				E'|t-s| +  \tau (\bE_f |x|^q)^{\frac1q} \, , 
	\end{align*}
for some $E'=E'(E,\bfmu, \theta)$, where the second inequality follows using the admissible transport plan $\de \pi(x,y):= f_\tau(y-x) \de y \de \mu_s(x)$. 
	Using the convexity of $\bW_q$ and the latter bound, by \eqref{eq:prop_f_g} 
	\begin{align*}
		\bW_q(\mu_t, \hat \mu_t)
			\leq
				\int g_\tau(t-s) \bW_q(\mu_t, f_\tau * \mu_s) \de s 
			\leq
				C \tau  \, ,
	\end{align*}
for some $C\in\R_+$ depending on $f$, $g$, and $E'$. 
By Lemma \ref{lemma:rescaling_error}, we get that
\begin{align*}
			\diam(I_\eps[\hat \mu_t])^{p-q}	\bW_q(\hat \mu_t, \cC_\eps \hat \mu_t)^q
		\leq 
			\int_{I_\eps[\hat \mu_t]^c} |x - x_{\hat \mu_t}|^p \de \hat \mu_t(x) \, .
\end{align*}
Collecting the various estimates, we end up with
\begin{align*}
	\diam(I_\eps[\hat \mu_t])^{p-q}\bW_q(\mu_t, \tilde \mu_t)^q
		\leq
		\bigg(
			\diam(I_\eps[\hat \mu_t])^{\frac{p}{q}-1} C \tau  + \left( \int_{I_\eps[\hat \mu_t]^c} |x - x_{\hat \mu_t}|^p \de \hat \mu_t(x) \right)^{\frac1q}
	\bigg)^q \, .
\end{align*}

Employing Lemma \ref{lemma:stability_Ieps}, together with Remark \ref{rem:time_reg_finite_energy}, and taking the limit as $\tau \to 0$,  for every fixed $\eps>0$ and $t \in [0,1]$ we find that
\begin{align*}
	\limsup_{\tau \to 0} 
			\diam(I_\eps[\hat \mu_t])^{p-q}\bW_q(\mu_t, \tilde \mu_t)^q
			&\leq 
			C' \left( 
			\int_{I_\eps[\mu_t]^c} |x|^p \de\mu_t(x)  + \eps 
			\right)
	 	\\
	 		&= 
	C'
	 			\left(
	 			\int_0^\eps 
	 				\Big|
	 					X_{ \mu_t}(z)
	 				\Big|^p
	 			\de z
	 		+
	 			\int_{1-\eps}^1 
	 				\Big|
	 					X_{ \mu_t}(z)
	 				\Big|^p
	 			\de z
	 			 + \eps
	 			 \right) \, , 
\end{align*}
for some $C'\in\R_+$ independent of $\eps$ and $t$. The latter term is infinitesimal as $\eps \to 0$ uniformly in time, thanks to Remark~\ref{rem:time_reg_finite_energy} and \eqref{eq:Wass_pseudo}. Therefore, we reach any threshold $\eta>0$ if we suitably choose first $\tau$  and then $\eps$ small enough.

	\medskip
\noindent
\textit{Step 5: Estimate of the tail.} \ 
We are left to prove we can choose $\tau,\eps>0$ such that
\begin{align*}
	\sup_{t\in[0,1]}
		\int_{\supp(\tilde \mu_t)^c} |x|^p \de \mu_t(x)			
		\leq \eta \, ,
\end{align*}
for any given threshold $\eta>0$. We claim that $I_{2\eps}[\mu_t] \subset I_\eps[\hat \mu_t]$, for $\tau$ sufficiently small. This would allow us to conclude, as in the proof of Step 4, that
\begin{align*}
		\sup_{t\in[0,1]} 
			\int_{I_\eps[\hat \mu_t]^c} |x|^p \de \mu_t(x)	
		\leq 
			\sup_{t\in[0,1]}
				 \int_{I_{2\eps}[\mu_t]^c} |x|^p \de \mu_t(x) \to 0	
\end{align*}
as $\eps \to 0$. More precisely, we shall prove that 
\begin{align}
		\label{eq:claim_delta_tau}
	I_{\eps + \delta_{\tau}}[\mu_t] \subset I_\eps[\hat \mu_t], 
		\quad 
	\forall t \in [0,1] ,
\end{align}
where $\delta_\tau$ can be explicitly chosen depending only on $\tau$, $p$, $M$, and the curve $\bfmu$  as
\begin{align*}
	\delta_\tau := 
	\Big(
	\sup_{t\in[0,1]}
		\bW_p(\mu_t,\hat \mu_t)^{\frac{p}{p+1}}
	\Big)
	\Big(
		(p+1)M^p
	\Big)^{\frac1{p+1}}.
\end{align*}
Using the fact that $\delta_\tau \to 0$ as $\tau \to 0$, the claim is proved.
To prove \eqref{eq:claim_delta_tau}, suppose that $\delta>0$ is such that
$
	X_{\mu_t}(\eps + \delta) < X_{\hat \mu_t}(\eps)
$, for some $t \in [0,1]$. Using the fact that $\rho_t\leq M$, we obtain
\begin{align*}
	X_{\mu_t}(\eps+\delta) > X_{ \mu_t}(\eps+z) + M^{-1}(\delta - z),
		\quad
	\forall z \in [0,\delta] .
\end{align*}
Hence, by assumption on $\delta$, for $z \in [0,\delta]$ we have that $$X_{\hat \mu_t}(\eps+z) - X_{\mu_t}(\eps+ z) > X_{ \mu_t}(\eps+\delta) - X_{\mu_t}(\eps+ z)  > M^{-1}(\delta - z) \, .$$ By Remark \ref{rem:Wass_pseudo}, we conclude that 
\begin{align*}
	\bW_p(\mu_t, \hat \mu_t)^p \geq \int_0^{\delta} 
	\Big|
		X_{\mu_t}( \eps + z) - X_{\hat \mu_t}(\eps + z)
	\Big|^p
		\de z
	> 
		\delta^{p+1} (p+1)^{-1}M^{-p},
\end{align*}
which shows that $\delta < \delta_\tau$. This shows that $X_{\mu_t}(\eps + \delta_\tau) \geq X_{\hat \mu_t}(\eps)$. In a similar way, one can prove that $X_{\mu_t}(1-\eps - \delta_\tau) \leq  X_{\hat \mu_t}(1-\eps)$, and thus proving \eqref{eq:claim_delta_tau}.
\end{proof}

\section{Proof of the main result: $\Gamma$-convergence.}
\label{sec:proof_dist}
In this section we prove Theorem \ref{thm:Gamma-conv}.
We start showing that for any sequence $\seq{x^N \in \cK^N}_N$, the weak limits of the associated empirical measures coincide with the weak limits of the corresponding piecewise-constant measures.

\begin{lemma}[Weak limits of empirical and piecewise-constant measures]
	\label{lemma:weak_limits}
	Let $p \in [1,\infty)$ and $\seq{x^N \in \cK_N}_N$. Set 
	$
	\hat{\mu}^N := \E^N(x^N)
	$, 
	$\mu^N := \cPC^N(x^N)
	$, and fix $\mu \in \cP(R)$.
	\begin{enumerate}[(i)]
	\setlength{\itemsep}{.2\baselineskip}
		\item 
		\label{i:lemma:weak_limits_1}
		$\hat \mu^N \to \mu$ vaguely if and only if $\mu^N \to \mu$ vaguely.
		\item 
		\label{i:lemma:weak_limits_2}
		The $p$th-moments of $\hat{\mu}^N$ are equi-integrable if and only if the $p$th-moments of $\mu^N$ are equi-integrable.
		\item
		\label{i:lemma:weak_limits_3}
		$\bW_p(\hat{\mu}^N, \mu) \to 0$  if and only if $\bW_p(\mu^N, \mu) \to 0$.
	\end{enumerate}
\end{lemma}

\begin{proof}

Note that \eqref{i:lemma:weak_limits_3} is a direct consequence of \eqref{i:lemma:weak_limits_1} and \eqref{i:lemma:weak_limits_2}. For the sake of simplicity, we omit the dependence on $N$ and write $x_i = x_i^N$. 

	\smallskip
\noindent 
\eqref{i:lemma:weak_limits_1}: \  It suffices to prove that $\hat{\mu}^N - \mu^N \to 0$ vaguely.
In fact, we claim that
\[ \hat{\mu}^N -   \frac{1}{N}  \sum_{i=0}^{N-1} \delta_{x_i} \to 0 \quad \mbox{ and } \quad   \frac{1}{N}  \sum_{i=0}^{N-1} \delta_{x_i} - \mu^N \to 0
\quad \text{vaguely}\, .  \]
For simplicity of notation,  we denote $\bar{\mu}^N := \frac{1}{N}  \sum_{i=0}^{N-1} \delta_{x_i}$ (note that in general $\bar{\mu}^N$ does not coincide with $\hat{\mu}^{N-1}$ because the set of particles may differ). 
For any $\vphi \in C^\infty_c(\setR)$, we denote by $\supp \vphi$ its (compact) support.  
On one hand
\begin{align*}
	\left| \int_\setR \varphi \de ( \hat{\mu}^N - \bar{\mu}^N) \right| & = \left|  \frac{1}{N+1} \sum_{i=0}^N \varphi(x_i) - \frac{1}{N} \sum_{i=0}^{N-1} \varphi(x_i) \right|  \\
	& \leq \frac{1}{N(N+1)} \sum_{i=0}^{N-1} |\varphi(x_i)| + \frac{\|\varphi(x_N)\|}{N+1} \leq \frac{2\|\varphi\|_{\infty}}{N+1} \to 0 \, .
\end{align*}
On the other side, 
\begin{align*}
	\left| \int_\setR \varphi \de  (\bar{\mu}^N - \mu^N) \right| & = \frac{1}{N} \sum_{i=0}^{N-1} \left| \varphi(x_i) - \fint_{x_i}^{x_{i+1}} \varphi(x) \de x \right|  \leq \frac{\| \vphi' \|_{\infty} }{N} \text{diam(spt}\vphi) \to 0 \, . 
\end{align*}
	\eqref{i:lemma:weak_limits_2}: \  Let us first assume that the $p$th-moments of $\hat \mu^N$ are equi-integrable, namely for every $\eps>0$ there is $L_\eps>0$ such that 
	\[ \limsup_{N \to \infty} \int_{\setR \setminus [-L_\eps, L_\eps]} |x|^p \de \hat{\mu}^N(x)  < \eps.  \]
	 We set 
	$
	i_\ell := \min \{ i : x_i \in [-L_\eps,  L_\eps]  \}$, $   i_r := \max \{ i : x_i \in [-L_\eps,  L_\eps]  \}. 
	$
	In particular
	\[\limsup_{N \to \infty} \frac{1}{N+1} \left( \sum_{i=0}^{i_\ell - 1} |x_i|^p +  \sum_{i=i_r + 1}^{N}  |x_i|^p \right) < \eps, \]
	from which, by monotonicity of the function $|\cdot|^p$ separately in $(\infty,-L_\eps)$ and $(L_\eps, \infty)$,  for $N$ big enough we deduce 
	\begin{align*}
	\int_{\setR \setminus [-L_\eps, L_\eps]} |x|^p \de \mu^N & \leq \frac{1}{N} \left(  \sum_{i=0}^{i_\ell - 1}  \fint_{x_i}^{x_{i+1}}  |x|^p \de x + \sum_{i=i_r}^{N-1} \fint_{x_i}^{x_{i+1}}  |x|^p \de x \right) \\
	& \leq  \frac{1}{N} \left( \sum_{i=0}^{i_\ell - 1} |x_i|^p +  \sum_{i=i_r}^{N-1} |x_{i+1}|^p  \right)
	\leq 2\eps \, ,
	\end{align*}
	thus providing the equi-integrability of the $p$th-moments of $\mu^N$.
	
	Viceversa: assume that for $\eps >0$, there exists $L_\eps >0$ such that 
	\[ \limsup_{N \to \infty} \int_{\mathbb{R} \setminus [-L_\eps, L_\eps]} |x|^p \de \mu^N < \eps.  \]
	Observe that,  calling again $i_\ell, i_r$ the first and the last particle inside $[-L_\eps,L_\eps]$ respectively,  for $N$ big enough we get
	\begin{align*}
	\int_{\setR \setminus [x_{i_\ell}, x_{i_r}]} |x|^p \de \mu^N  &= \int_{\setR \setminus [-L_\eps, L_\eps]} |x|^p \de \mu^N + \int_{[-L_\eps, L_\eps] \setminus [x_{i_\ell}, x_{i_r}]} |x|^p \de \mu^N \\
	& \leq \eps + \frac{1}{N} \left( \frac{1}{\Delta x_{i_\ell-1} } \int_{-L_\eps}^{x_{i_\ell}} |x|^p \de x      \frac{1}{\Delta x_{i_r}} \int^{L_\eps}_{x_{i_r}} |x|^p \de x     \right) 
	\leq \eps + \frac{2|L_\eps|^p}{N} \leq 2\eps \, .
	\end{align*}
	Therefore
	\begin{equation}\label{eq:equint p-moments PC}
	\limsup_N \frac{1}{N} \left(  \sum_{i=0}^{i_\ell - 1} \fint_{x_i}^{x_{i+1}} (-x)^p \de x + \sum_{i=i_r}^{N - 1}   \fint_{x_i}^{x_{i+1}} x^p \de x  \right) < 2\eps.
	\end{equation}
	Since $|x_{i_\ell -1}| \leq |x_{i_\ell -2}|$ and $a^{p+1}-b^{p+1} \geq (a-b) a^p$ for $0<b<a$, we have that
	\begin{align*}
		\int_{\setR \setminus [-L_\eps, L_\eps]} |x|^p \de  \hat{\mu}^N
			&\leq  
		2\left(     \sum_{i=0}^{i_\ell - 2} |x_i|^p  +  \sum_{i=i_r}^{N - 1} |x_{i+1}|^p  \right)
	\\
			&\leq 2(p+1)
			\left(
		\sum_{i=0}^{i_\ell - 1} \fint_{x_i}^{x_{i+1}} (-x)^p \de x +  \sum_{i=i_r}^{N - 1}   \fint_{x_i}^{x_{i+1}} x^p \de x
		\right) \, .
	\end{align*}
	Hence, by \eqref{eq:equint p-moments PC} we conclude the equi-integrability of the $p$-moments of $\hat{\mu}^N$.
\end{proof}

\subsection{Proof of the $\Gamma$-limsup inequality}

In the first result, we show the $\Gamma$-$\limsup$ inequality for regular curves, in the sense of Definition~\ref{def:regular_measures}, and we conclude using the approximation result provided by Proposition \ref{prop:approximation}.  
We emphasise that, for regular curves, we are able to construct a recovery sequence satisfying a discrete version of Definition~\ref{def:regular_measures}. In principle, this is not necessary to prove Theorem~\ref{thm:Gamma-conv}, nonetheless it plays a crucial role in Section~\ref{sec:JKO} for the convergence of the JKO-schemes. Throughout the whole section, we fix $\phi_{p,m}$ as in Definition~\ref{def:action_density_functions}, with $1 < p < \infty$. Moreover, we denote by $\bW_\infty$ the distance given by $\bW_\infty(\mu_0, \mu_1) := \| X_{\mu_0} - X_{\mu_1} \|_\infty$, for $\mu_0,\mu_1 \in \cP(\R)$ compactly supported.

\begin{proposition}[Existence of (regular) recovery sequences for regular curves]
	\label{prop:upper_regular}
	 Let $(\tilde \bfmu,\tilde \bfnu) \in \CE$ be regular curves in the sense of Definition \ref{def:regular_measures} for some compact set $K$ and constants $0 < l \leq \tilde M < M$. 
	 Then there exist discrete curves $\seq{\tilde x^N \in C^\infty ([0,1] ;  \cK_{N} )}_N$ such that
	\begin{enumerate}[(i)]
		\setlength{\itemsep}{.2\baselineskip}
		\item \label{i:prop:upper_regular_1} 
		the measures
		$\displaystyle 
		\cPC^N
		(	\tilde x^N(k)    )$ (and hence $\E^N( 	\tilde x^N(k)    )$ as well, by Lemma \ref{lemma:weak_limits}) converge to
	    $\tilde \mu_k$ in $\bW_q
		$ as $N \to \infty$, for $k=0,1$, for every  $q \in (1,+\infty]$.
		\item 
		\label{i:prop:upper_regular_2} 
		$\displaystyle
		\limsup_{N \to \infty}
		\int_0^1 \Phi_{p,m}^N
		\big(  
		\tilde x^N (t) , \dot{\tilde x}^N(t) 
		\big) \de t
		\leq
		\bpW(\tilde \mu_0,\tilde \mu_1)^p$ .
		\item 
		\label{i:prop:upper_regular_3} 
		The following regularity properties holds: for $i=0, \dots, N$, $t \in [0,1]$, $N \in \setN$
		\begin{align*}
			l \leq \tilde R_i^N(\tilde x^N(t)) \leq \tilde M  , \quad 
			 \supp  \cPC^N
		\big( 	\tilde x^N(t) \big) = \supp(\tilde \mu_t) \subset K \, .
		\end{align*}
	\end{enumerate}
\end{proposition}

\begin{proof}
	The main idea is to consider the pseudo-inverse associated with the regular curve $\tilde \bfmu$ and discretise them in space on a scale $1/N$. 
	
	\medskip
	\noindent
	\textit{Step 1: Discretisation of $(\tilde \mu_t, \tilde \nu_t)$}. \
	Denote by $\tilde v_t = \tilde j_t/\tilde \rho_t$ the associated velocity field, and by $I_t$ its compact and connected support. By construction, we have that $\big\{ t \mapsto \tilde v_t \in C^\infty(I_t) \big\} \in C^\infty([0,1]) $. In particular, the curve $\tilde X_t := X_{\tilde \mu_t}$ solves the differential equation (see \eqref{eq:equation_pseudoinverse_time})
	\begin{align*}
		\dot{\tilde X}_t(z) = \tilde v_t(\tilde X_t(z)) \, , \quad  \tilde X_0 := X_{\tilde \mu_0}, \quad \tilde X_0 \in C^1([0,1]) \, ,
	\end{align*}
	where the latter regularity follows from the continuity of (the density of) $\mu_0$ and the fact that $I_t$ is compact and connected.
	Moreover, by the homogeneity of $\phi_{p,m}$, we have that for $t \in [0,1]$
	\begin{align}	
 \label{eq:equality_optimum}
			\int_0^1 
			( \tilde \rho_t \circ  \tilde X_t)^{-1} 	
			\phi_{p,m}( \tilde \rho_t \circ  \tilde X_t,  \tilde j_t \circ  \tilde X_t) 
			\de z 
			&= 
			\int_\setR 
			\phi_{p,m}( \tilde \rho_t, \tilde j_t) 
			\de x 
	\end{align}	
	We define for every $N\in \setN$ the discrete curve $ \tilde x^N(t)$ as  
	\begin{align*}
		\tilde x_i^N(t) :=  \tilde X_t \left( \frac{i}{N} \right), \quad t \in [0,1], \; i \in \{ 0, ..., N \} \, ,
	\end{align*}
	and consider the corresponding measures $ \tilde  \mu_t^N=\cPC^N(\tilde x^N(t)) \in \cP(\R)$ with densities
	\begin{align}	\label{eq:defrecovery}
		\tilde \rho_t^N(x) =  \sum_{i=0}^{N-1} \tilde R_i^N \mathbbm{1}_{\left[   \tilde x_i^N,  \tilde x_{i+1}^N \right)}(x)
		\, , \quad 
		\tilde R_j^N:= R_j^N(\tilde x^N) \, , 
	\end{align}
	for $x \in \setR$ and $j = 0, \dots , N-1$.
	Note that the pseudo-inverse of $ \tilde  \mu_t^N$ is nothing but the piecewise affine interpolation of the $ \tilde X_t$ on $[0,1]$. 
	Thus for $i =0, \dots, N-1$
	\[
	\tilde X_{ \tilde  \mu_t^N}(z) = (Nz-i) \tilde X_t \left( \frac{i+1}{N} \right) + (i+1-Nz)  \tilde X_t \left( \frac{i}{N} \right)
		\qquad 
	\text{for } z \in \left[ \frac{i}{N}, \frac{i+1}{N} \right].
	\]
	
	From the continuity of $ \tilde X_t$, it follows that
	$
		\sup_{t \in [0,1]}
		\big\|  \tilde X_t -  \tilde X_{ \tilde  \mu_t^N} \big\|_{\tL^\infty(0,1)} \to 0
	$
	as $N \to \infty
$,
	which shows \eqref{i:prop:upper_regular_1} (see also Remark \ref{rem:Wass_pseudo}).
	By construction, for any $t \in [0,1]$ the $\supp \tilde  \mu_t^N = \supp  \tilde  \mu_t$ for every $N \in \N$,  in particular $\supp \tilde  \mu_t^N \subset K$ and also $\tilde x_i^N(t) \in K$ for every $i = 0, \dots , N$.  Also,  being $ l \leq \tilde \rho_t \leq \tilde M$,  it is immediate to see that $l \leq \tilde R_i^N \leq  \tilde M$,  thus proving \eqref{i:prop:upper_regular_3}.

	\medskip
	\noindent
	\textit{Step 2: Energy estimates}. \
	Using that $( \tilde X_t)_{\#} \de z =  \tilde  \mu_t$, we deduce that
	\begin{align*}
		\fint_{\frac iN}^{\frac{i+1}{N}}
		\tilde R_i^N (\tilde \rho_t \circ  \tilde X_t)^{-1}(z)  \de z
		= 
		N \tilde R_i^N \int_{ \tilde X_t( \frac iN ) }^{ \tilde X_t ( \frac {i+1}{N} ) } \tilde \rho_t^{-1}(x) \de  \tilde  \mu_t(x)  = 1 \, , 
	\end{align*}
	by the very definition of $\tilde R_i^N$. This means that the measures $\de \xi_t^{N,i} := N \tilde R_i^N (\tilde \rho_t \circ  \tilde X_t)^{-1}(z) \de z \in \cP \big( \big[ \frac iN, \frac{i+1}{N} \big] \big)$.  
 By Jensen's inequality and \eqref{eq:equality_optimum} 
	\begin{align} \label{eq:lowerbound_Jensen}
		\nonumber
		\int_\setR \phi_{p,m}(\tilde \rho_t, \tilde j_t) \de x 
		&\geq \frac{1}{N} \sum_{i=0}^{N-1} 
		( \tilde R_i^N )^{-1} \phi_{p,m} 
		\bigg( 
		\int_{\frac iN}^{\frac{i+1}{N}} 
		( \tilde \rho_t \circ  \tilde X_t,  \tilde j_t \circ  \tilde X_t) \de \xi_t^{N,i}(z)
		\bigg) \\ \nonumber
		&= \frac{1}{N} \sum_{i=0}^{N-1} 
		( \tilde R_i^N )^{-1} \phi_{p,m} 
		\bigg( 
		\tilde R_i^N , \tilde R_i^N  \fint_{\frac iN}^{\frac{i+1}{N}} v_t \circ  \tilde X_t (z) \de z 
		\bigg) \\ 
		&= \frac{1}{N} \sum_{i=0}^{N-1} 
		( \tilde R_i^N )^{p-1} \phi_{p,m} 
		\bigg( 
		\tilde R_i^N, \fint_{\frac iN}^{\frac{i+1}{N}} \partial_t \tilde X_t (z) \de z 
		\bigg) \, ,
	\end{align}
	where at last we used the $p$-homogeneity of $\phi_{p,m}$ with respect to the second variable. Moreover, by $\tilde X \in C^1([0,1] \times [0,1])$,  $\text{Im}( \tilde X_t) \subset \supp \tilde \rho_t $, and the boundedness of $\tilde v$,
	\begin{align*}
		\left\{
			(\tilde R_i^N \big( \tilde x_t^N \big), \dot{\tilde X}_t(z)) \suchthat 
				t>0 \, , \; z \in [0,1]\, , \; N \in \setN 
		\right\}
			\subset
		[l,\tilde M] \times [0, \|  \tilde v \|_\infty] \, =K
	\end{align*}
		It is clear that $\phi_{p,m}$ is Lipschitz in $K$, therefore,  using \eqref{eq:def_PhiN}, \eqref{eq:lowerbound_Jensen} and \eqref{eq:ass_rho*}, we get that there exists a constant $C \in \R_+$ which only depends on $l$, $\tilde M$, $p$, $\theta$, and $\rho^*$ such that
	\begin{align*}
		\int_0^1 \Phi_{p,m}( \tilde \rho_t, \tilde j_t) \de t  
		\geq 
		\int_0^1 
		\Phi_{p,m}^N( \tilde x^N, \dot{ \tilde x}^N) \de t
		- 
		\frac CN \Big(  \|\tilde X \|_{C^2}
		+  
		\| \tilde X \|_{C^1}^p  \Big) \, .
%
	\end{align*} 
	Letting $N \to +\infty$ we obtain \eqref{i:prop:upper_regular_2}. 
\end{proof}

We are finally ready to prove the $\Gamma$-$\limsup$ inequality in Theorem \ref{thm:Gamma-conv}.

\begin{proof}[Proof of Theorem \ref{thm:Gamma-conv}, $\Gamma$-$\limsup$]
Fix $\mu_0,\mu_1 \in \Prob_q(\setR)$ be such that $\bpW(\mu_0,\mu_1)<+\infty$. Let $(\bfmu,\bfnu) \in \CE(\mu_0,\mu_1)$ be a $\bpW$-geodesic between $\mu_0$ and $\mu_1$ (see \eqref{i:geodesics_convexity}).
For every $\eta>0$, we apply Proposition \ref{prop:approximation} and Proposition \ref{prop:upper_regular}, and find regular curves $\tilde \mu_0^\eta$, $\tilde \mu_1^\eta$ (thus belonging to $\Prob_q(\setR)$)  and discrete curves $\tilde x^{N,\eta} \in W^{1,p}(0,1 ; \cK_N)$ such that $	\bW_q(\mu_k, \tilde \mu_k^\eta)^q \leq \eta$, $ 
\E^N(\tilde x_k^{N,\eta}) \xrightarrow{\bW_q} \tilde \mu_k^\eta$, and 
\begin{gather*}
	\limsup_{N \to \infty}
		\int_0^1 \Phi_{p,m}^N
		\big(  
			\tilde x^{N,\eta} (t) , \dot{\tilde x}^{N,\eta}(t) 
		\big) \de t
	\leq
	\bpW(\tilde \mu_0^\eta,\tilde \mu_1^\eta)^p
	\leq
		\bpW(\mu_0,\mu_1)^p + \eta \, .	
\end{gather*} 
We conclude the proof by choosing a suitable diagonal sequence $\tilde x^N:= \tilde x^{N,\eta(N)}$.
\end{proof}

\subsection{Proof of the $\Gamma$-liminf inequality}
In this section we prove the $\Gamma$-$\liminf$ in Theorem \ref{thm:Gamma-conv}. A crucial tool is a compactness result for "almost" solutions to the continuity equation.
This is the content of the next Proposition, which is inspired by \cite[Lemma 4.5]{dolbeault2012} (see also \cite[Proposition 2.5]{lisiniMarigonda2010}) and follows similar ideas. For the sake of clarity, we provide a proof.
\begin{definition}[Asymptotic solutions to $\bCE$]
	We say that a sequence $\seq{\mu_t^N,\nu_t^N}_N$ is  \textit{asymptotically a solution to $\bCE$} on $[0,1]$ if the same assumptions of Definition \ref{def:CE} are satisfied for every $N \in \setN$, but with $3.$ replaced by
	\begin{itemize}
		\item[$3'$.] For every $\vphi \in C_c^1(\setR)$, the measurable maps $\cR^N \vphi:= t \mapsto (\partial_t \mu_t^N + \partial_x \nu_t^N, \vphi)$ converge to $0$ weakly in $\tL^1(0,1)$ as $N \to \infty$.
	\end{itemize}
We use the notation $\cR_t^N \vphi:= (\cR^N \vphi)(t)$, for $t \in (0,1)$.
\end{definition}
\begin{proposition}[Compactness]
		\label{prop:compactness_asymptCE}
	Let $\seq{\mu_t^N, \nu_t^N}_N$ be asymptotically a solution to $\bCE$ satisfying
	\begin{align*}
		A:=\sup_{N \in \setN} 
			\int_0^1 
				\Phi_{p,m}(\mu_t^N, \nu_t^N) \de t 
			< \infty \, .
	\end{align*} 
Suppose that $\mu_0^N \to \mu_0 \in \cP(\R)$ vaguely. Then there exists a (non-relabeled) subsequence and a solution $(\mu_t,\nu_t) \in \bCE$ such that $\mu_t^N \to \mu_t$ vaguely in $\Prob(\setR)$ for all $t \in [0,1]$, $\bfnu^N \to \bfnu$ vaguely in $\cM((0,1) \times \setR)$, and 
	\begin{align*}
		\liminf_{N \rightarrow +\infty}
			\int_0^1 
			\Phi_{p,m}(\mu_t^N, \nu_t^N) \de t 
		\geq 
		\int_0^1 \Phi_{p,m}(\mu_t, \nu_t) \de t . 
	\end{align*}
		
\end{proposition}

\begin{proof}
	We denote by $\rho_t^N$, $j_t^N$ the corresponding densities of $\mu_t^N$ and $\nu_t^N$ with respect to the Lebesgue measure on $[0,1]\times \setR$. Then an application of the H\"{o}lder inequality yields for every Borel set $B\subset \setR$ and $t_1<t_2 \in [0,1]$ the bound 
	\begin{align}	
	\begin{aligned} \label{eq:equiintegr_jN}
	\int_{t_1}^{t_2} \int_B | j_t^N | \de x \de t 
	& \leq \left( \int_{t_1}^{t_2} \int_B m(\rho_t^N) \de x \de t \right)^{\frac{1}{p'}}  \left( \int_0^1 \Phi_{p,m}(\mu_t^N, \nu_t^N) \de t \right)^{\frac{1}{p}} \\
	& \leq \left( \| m \|_{\infty} \Leb^1(B) (t_2-t_1)  \right)^{\frac{1}{p'}}A^{\frac{1}{p}}
	\end{aligned}
	\end{align}
	for any $N \in \setN$, for $p'$ conjugate of $p$ (note that $m \in \tL^\infty(0,M)$ by concavity and $m(0)=0=m(M)$). Consequently there exists $\bfnu \in \cM([0,1]\times \setR)$ and a subsequence (not relabeled) such that $\bfnu^N \to \bfnu$ vaguely as $N \rightarrow +\infty$.

	 Moreover, \eqref{eq:equiintegr_jN} shows that the map $t \mapsto |\nu_t^N|(B)$ is equi-integrable in $[0,1]$ and uniformly bounded in $L^1(0,1)$, for every $B \subset \setR$. By Dunford--Pettis \cite[Theorem~4.30]{Brezis:2011}, this in particular implies that $\bfnu$ satisfies  for every $R>0$ 
	\begin{align}	\label{eq:L1_decomp_nu}
	|\bfnu|(B_R \times I) = \int_I m_R(t) \de t, \quad \forall I \subset([0,1), \, B_R \subset \setR, \quad \text{with }m_R \in \tL^1(0,1),
	\end{align}
	so that we can write $\bfnu=\int_0^1 \nu_t \de t$, intended as in \eqref{eq:bnu}, for a Borel family of (signed) finite measures $\seq{ \nu_t }_{t \in (0,1)}$ (see \cite[Lemma~4.5]{dolbeault2012} for a similar argument). 
	
	By assumption we know that $\mu_0^{N} \to \mu_0 \in \cP(\R)$ vaguely (in fact narrowly, by Remark~\ref{rem:time_reg_finite_energy}).  
	We now claim that for every time $t\in [0,1]$,  the sequence $\mu_t^{N}$ vaguely converges to some $\mu_t \in \cP(\R)$, as a consequence of $(\mu_t^N, \nu_t^N)$ being asymptotically a solution to $\bCE$.
	
	Indeed, consider a function $\vphi \in C_c^1(\setR)$ and define $\phi(\tau,x):= \mathbbm{1}_I(\tau) \partial_x \vphi(x)$, $I:=[0,t]$.  Note that the discontinuity set of $\phi$ is concentrated on $D:=\{0,t\} \times \setR$ and from \eqref{eq:L1_decomp_nu} one has $|\bfnu|(D)=0$. It follows \cite[Proposition 5.1.10]{AmbrosioGigliSavare08} that
	\begin{align*}
	\lim_{N \rightarrow +\infty} \int_I \int_\setR \partial_x \vphi \de \nu_\tau^{N} \de \tau
	= \lim_{N \rightarrow +\infty} \int_{I \times \setR} \phi \de \bfnu^{N}
	= \int_{I \times \setR} \phi \de \bfnu
	= \int_I \int_\setR \partial_x \vphi \de \nu_\tau \de \tau. 
	\end{align*}
	
	Moreover, by definition of $\cR_\tau^{N}$ and integrating over $I$ we obtain
	\begin{align}	
	\begin{aligned} \label{eq:samelimit_mutN}
	\lim_{N \rightarrow +\infty} \int_\setR \vphi \de \mu_t^{N} 
	&= \lim_{N \rightarrow +\infty} \left( \int_\setR \vphi \de \mu_0^{N} + \int_0^t \int_\setR \partial_x \vphi \de \nu_\tau^{N} \de \tau + \int_0^t \cR_\tau^{N} \vphi \de \tau \right) \\
	&= \int_\setR \vphi \de \mu_0 + \int_I \int_\setR \partial_x \vphi \de \nu_\tau \de \tau \, ,
	\end{aligned}
	\end{align}
	where we used that $\mu_0^{N} \to \mu_0$ vaguely and that $\cR^{N} \vphi \to 0$ weakly in $\tL^1(0,1)$. 
Hence, \eqref{eq:samelimit_mutN} shows that $\mu_t^{N}$ is vaguely converging to some $\mu_t\in \cP (\setR)$ as $N \rightarrow +\infty$ for every $t \in [0,1]$, with  $(\mu_t,\nu_t) \in \bCE$.
	Finally, the lower semicontinuity estimate follows from the lower semicontinuity of $\Phi_{p,m}$, see \eqref{i:lsc}. 
\end{proof}

We are ready to prove the $\Gamma$-$\liminf$ inequality in Theorem \ref{thm:Gamma-conv}.

\begin{proof}[Proof of Theorem \ref{thm:Gamma-conv}, $\Gamma$-$\liminf$]
Let $\mu_0^N=\cPC^N(x_0^N)$, $ \mu_1^N=\cPC^N(x_1^N)$ be two sequences in $\cP(\setR)$ converging vaguely to $\mu_0,\mu_1$, respectively.  We claim that
\begin{align}	\label{eq:liminf_ineq}
	\liminf_{N \rightarrow +\infty} \bNW(\mu_0^N, \mu_1^N) \geq \bpW(\mu_0,\mu_1).
\end{align}
Applying Lemma~\ref{lemma:geodesics_discrete}, with no loss of generality we can write 
\begin{align}	\label{eq:unif-bound-li}
	\begin{aligned}
	\bNW(\mu_0^N, \mu_1^N)^p = \int_0^1 \Phi_{p,m}^N(x^N(t), \dot{x}^N(t)) \de t
\ , \ \ 
	E:=\sup_{N \in \setN} \sup_{t \in [0,1]}
		\Phi_{p,m}^N(x^N(t), \dot{x}^N(t))  < \infty \,  ,
	\end{aligned}
\end{align}
where $x^N \in W^{1,p}(0,1 ; \cK_N)$ is the discrete geodesic between $x_0^N$ and $x_1^N$ with respect to $d_{p,m}^N$.
We define $\mu_t^N:= \cPC^N(x_t^N)$ with density (w.r.t the Lebesgue measure on $\setR$) $\rho_t^N$ given by the piecewise-constant map
\begin{align}	\label{eq:defmuN}
	\rho_t^N(x) =  \sum_{i=0}^{N-1} R_i^N(t) \mathbbm{1}_{\left[  x_i^N(t), x_{i+1}^N(t) \right)}(x) \, , 
\end{align}
and define the momentum fields as
\begin{align}	\label{eq:defjN}
	j_t^N(x):= \sum_{i=0}^{N-1} \dot{x}_i^N(t) R_i^N(t) \mathbbm{1}_{\left[x_i^N(t), x_{i+1}^N(t)\right)}(x), 
	\quad 
	\nu_t^N := j_t^N \Leb^1 .
\end{align}

\medskip
\noindent
\textit{Step 1: Uniform error estimate.} \, Let us pick any $\vphi \in C_c^{\infty}(\setR)$. We claim that
\begin{align}	\label{eq:error-CE}
	\sup_{t \in (0,1)} \left| \left( \partial_t \mu_t^N + \partial_x \nu_t^N, \vphi \right)  \right| \leq c N^{-\frac{1}{p'}}  \| \vphi \|_{C^2} \left(\diam (\supp \vphi) + 2 \right), 
\end{align}
for some $c=c(\theta,E,p) \in \setR^+$, where $p'$ is the conjugate of $p$. 
For simplicity, in this proof we use the shorthand notation
\begin{align*}
	\Delta x_i := x_{i+1}^N(t) - x_i^N(t), \	\quad \Delta \dot{x}_i :=  \dot{x}_{i+1}^N(t) - \dot{x}_i^N(t), \quad R_i:= R_i^N(t), \quad I_i := [x_i,x_{i+1}]  .
\end{align*}
In particular, we have that  $\dot{R_i} = -N R_i^2 \Delta \dot{x}_i \, $, from which we infer
\begin{align*}
\frac{\de}{\de t} \int \vphi \de \mu_t^N 	
	= -\sum_{i =0}^{N-1} R_i \Delta \dot{x}_i \fint _{I_i} \vphi \de x - R_i \big( \vphi(x_i)\dot{x}_i - \vphi(x_{i+1})\dot{x}_{i+1} \big) .
\end{align*}
Using a second order Taylor expansion for $\vphi$, we know there exist a family of points $\seq{ c_x }_x$ such that for every $x \in I_i$, we have $c_x \in (x, x_{i+1})$ and 
\begin{align*}
	\vphi(x) = \vphi(x_{i+1}) + \vphi'(x_{i+1})(x-x_{i+1}) + \frac12 \vphi''(c_x) (x-x_{i+1})^2 \, , 
\end{align*}
for $i=0, \dots, N-1$. It follows that
\begin{align*}
	\frac{\de}{\de t} \int \vphi \de \mu_t^N 	
	= \sum_{i=0}^{N-1} R_i \dot{x}_i \left( \vphi(x_{i+1}) - \vphi(x_i) \right) + \left( \cR_t^N , \vphi \right)
	= \left(\nu_t^N, \partial_x \vphi \right) + \left( \cR_t^N , \vphi \right) ,
\end{align*}
where the error $\cR_t^N \in \cD'(\setR)$ is given  by
\begin{align*}
	\left( \cR_t^N , \vphi \right) = \sum_{i=0}^{N-1} \frac{1}{2N}\Delta \dot{x}_i \vphi'(x_{i+1}) - \frac{1}{2} R_i \Delta \dot{x}_i \fint_{x_i}^{x_{i+1}} \vphi''(c_x) (x-x_{i+1})^2 \de x .
\end{align*}
 By $\theta \in \tL^{\infty}(0,M)$ and \eqref{eq:unif-bound-li}, we deduce the uniform bound
\begin{align}
		\label{eq:bound_timederivative_xi}
	\frac{1}{N+1} \sum_{i=0}^N |\dot{x}_i|^p\leq \| \theta \|_{\infty}^{p-1} \Phi_{p,m}^N(x, \dot{x}) \leq \| \theta \|_{\infty}^{p-1} E
		\quad \Longrightarrow \quad 
		\sup_i \left| \dot{x}_i \right| \leq c N^{\frac{1}{p}},
\end{align}
for some $ c=c(\| \theta \|_{\infty}, E, p) \in \setR^+$.
Consequently, on one hand by $x_i < x_{i+1}$ we get
\begin{gather*}
	\left| \frac{1}{N} \sum_{i=0}^{N-1} \Delta \dot{x}_i \vphi'(x_{i+1})  \right|
	\leq \frac{1}{N}  \sum_{i=0}^{N-1} | \dot{x}_i | |  \vphi'(x_i) - \vphi'(x_{i+1}) |  + \frac{1}{N} \Big(  | \dot{x}_0 \vphi'(x_0)| + | \dot{x}_N \vphi'(x_N) | \Big) \\
	\quad 
	\leq c N^{-\frac{1}{p'}} \left( \Lip \vphi' \sum_{i \in \Sph^N} \Delta x_i + 2 \Lip \vphi \right) \leq c N^{-\frac{1}{p'}}(\text{diam}(\supp \vphi)+2 )\| \vphi \|_{C^2} ,
\end{gather*}
where we set $\Sph^N:= \{ i \in \{ 0 , \dots, N-1 \} \suchthat I_i \cap \supp \vphi \neq \emptyset \}$.
On the other hand, using once more that $x_i < x_{i+1}$ and \eqref{eq:bound_timederivative_xi} we obtain
\begin{align*}
	&\left| \sum_{i=0}^{N-1} R_i \Delta \dot{x}_i \fint_{x_i}^{x_{i+1}} \vphi''(c_x) (x-x_{i+1})^2 \de x  \right|
		\leq \frac{\left\| \vphi \right\|_{C^2}}{N} \sum_{i\in \Sph^N} | \Delta \dot{x}_i | \Delta x_i 
	\\
		&\qquad\leq 2cN^{-\frac{1}{p'}} \left\| \vphi \right\|_{C^2} \sum_{i\in \Sph^N} \Delta x_i 
		\leq 2c N^{-\frac{1}{p'}} \left\| \vphi \right\|_{C^2} \text{diam}(\supp \vphi),
\end{align*}
which concludes the proof of \eqref{eq:error-CE}.

\medskip
\noindent
\textit{Step 2: Compactness.} \, 
Consider $\mu_t^N, \nu_t^N$, defined in \eqref{eq:defmuN}, \eqref{eq:defjN} and $\bfnu$ as in \eqref{eq:bnu}. We claim that up to subsequence, we have that 
\begin{align*}	
	 \bfnu^N \to \bfnu = \int_0^1 \nu_t \de t \in \cM((0,1) \times \setR) \quad \text{and} \quad 
	\mu_s^N \to \mu_s \in \cP(\setR) \quad \text{vaguely}, 
\end{align*}
for all $s \in [0,1]$, with $(\mu_t,\nu_t) \in \bCE$.
To show this, we prove that the assumptions of Proposition \ref{prop:compactness_asymptCE} are satisfied. The fact that $(\mu_t^N, \nu_t^N)_N$ is asymptotically a solution to $\bCE$ directly follows from \eqref{eq:error-CE}, the convergence $\mu_0^N \to \mu_0$ is already guaranteed by construction whereas the action bound follows from \eqref{eq:unif-bound-li} and \eqref{eq:A<AphiN}. We can then apply Proposition \ref{prop:compactness_asymptCE} and conclude.

\medskip
\noindent
\textit{Step 3: Action estimate}. \
We conclude the proof of \eqref{eq:liminf_ineq} applying the lower semicontinuous estimate provided by Proposition \ref{prop:compactness_asymptCE}. Precisely we have
\begin{align*}
	\liminf_{N \rightarrow +\infty} \bNW(\mu_0^N, \mu_1^N)^p
		&= \liminf_{N \rightarrow +\infty} \int_0^1\Phi_{p,m}^N(x^N, \dot{x}^N) \de t 
		\geq  \liminf_{N \rightarrow +\infty} \int_0^1 \Phi_{p,m}(\mu_t^N, \nu_t^N) \de t \\
		&\geq \int_0^1 \Phi_{p,m}(\mu_t, \nu_t) \de t \geq \bpW(\mu_0,\mu_1)^p \, , 
\end{align*}
where in the first inequality we used \eqref{eq:A<AphiN}, and at last $(\mu_t,\nu_t) \in \bCE$.
\end{proof}
 
\section{Proof of the main result: convergence of the JKO schemes}
\label{sec:JKO}
In this last section of the paper, we prove the discrete-to-continuum convergence of the JKO-schemes, i.e. Theorem \ref{thm:JKO_q}.

Fix $\phi_{2,m}$ as in Definition~\ref{def:action_density_functions} and $q \in [1,2)$. For a given initial sequence $\mu_N^{(0)} \in \E^N$ and $\mu^{(0)} \in \cP_2(\setR)$, recall the definition of the associated operators $\cJ_{\tau,N}$, $\cJ_\tau$ as given in \eqref{eq:def_JtauN}, \eqref{eq:def_Jtau}. In order to prove the convergence of the minimisers, our goal is to show a $\Gamma$-convergence result and an equi-coercivity property of the families $ \seq{\cJ_{\tau,N}}_N$ in the $\bW_q$ topology. 

Assuming that $\mu_N^{(0)} \to \mu^{(0)}$ in $\bW_q$, one of the major difficulties is to show that $\cJ_{\tau,N}$ $\Gamma$-converge to $\cJ_\tau$ as $N \to \infty$ despite the dependence on the quadratic (generalised) Wasserstein distance. The extra assumption \eqref{eq:unif_bounds_2ndmoms} of uniformly bounded second moments on $\seq{\mu_N^{(0)}}_N$ is crucial to this purpose. Also important is the fact that such assumption can be "iterated", namely from \eqref{eq:unif_bounds_2ndmoms} we can deduce a similar bound $\mu_N^{(1)}$, and iteratively for $\mu_N^{(n)}$, for every $n \in \setN$.

The proof of the $\Gamma$-convergence relies on an improved version of the $\Gamma$-$\limsup$ in Theorem \ref{thm:Gamma-conv}, which is the main technical difficulty of this last section.

\subsection{An improved limsup inequality}
The proof is based on the regular approximation result shown in Proposition~\ref{prop:approximation} and Proposition~\ref{prop:upper_regular}. 

\begin{proposition}[Limsup inequality, JKO]
		\label{prop:limsup_JKO}
	Fix $\mu_0$, $\mu_1 \in \Prob_2(\setR)$ and an approximating sequence $\cE^N \ni \mu_0^N \to \mu_0$ in $\bW_q$ for some $q>1$, satisfying \eqref{eq:unif_bounds_2ndmoms}. Then, for every $\eps>0$, there exists $\bar  \mu_1 \in \Prob_2(\setR)$ and a sequence $\seq{ \mu_1^N \in \E^N}_N$ such that
	\begin{enumerate}[(i)]
	\setlength{\itemsep}{.2\baselineskip}
		\item \label{i:lem:limsup_JKO_1}
		$\displaystyle
		 \mu_1^N \to \bar \mu_1
				\quad
			\text{in } 
				\big( \Prob_q(\setR), \bW_q \big) 
		$.
		\item \label{i:lem:limsup_JKO_2}
		$\displaystyle
			\bW_q(\mu_1,\bar \mu_1) \leq \eps
		$.
		\item \label{i:lem:limsup_JKO_3}
		$\displaystyle
			\limsup_{N \to \infty}
				\bW_{2,m}^N
					\big(
						\mu_0^N,  \mu_1^N 
					\big)
			\leq
				\btwoW(\mu_0, \mu_1) + \eps
		$.
	\end{enumerate}
\end{proposition}

\begin{remark}[Relation to Theorem~\ref{thm:Gamma-conv}, $\Gamma$-limsup]
	Note that Proposition~\ref{prop:limsup_JKO} is not a trivial consequence of the $\Gamma$-limsup in Theorem~\ref{thm:Gamma-conv}, due to the fact $\mu_0^N$ does not necessarily coincide with the recovery sequence for $\mu_0$ obtained in Theorem~\ref{thm:Gamma-conv}. 
\end{remark}
\begin{proof}
	Without loss of generality, we assume $\btwoW(\mu_0, \mu_1) < \infty$.
	Let $w^N \in \cK^N$ such that $\mu_0^N = \cE^N(w^N)$. Let $(\bfmu,\bfnu) \in \CE$ be the $\btwoW$-geodesic \eqref{i:geodesics_convexity} between $\mu_0$ and $\mu_1$. 
	
	\smallskip 
	\noindent
	{\underline{Case 1}}: \
	We first assume that 
	\begin{align}
		\label{eq:assumpt_Case1}
		\max
		\left\{ 
		\sup_{N \in \N} \sup_{i \leq N}
			R_i^N(w^N) 
		, \, 
			\sup_{t \in [0,1]} \left\| \rho_t \right\|_{L^\infty(\R)}
		\right\}
			= \bar M < M \, , 
	\end{align}
where $\rho_t$ denotes the density of $\mu_t$ w.r.t. $\Leb^1$.
In this case, for every $\eta>0$, we apply Proposition \ref{prop:approximation} and find regular curves $(\tilde \bfmu, \tilde \bfnu) \in \CE$ in the sense of Definition~\ref{def:regular_measures} with densities $
\tilde \rho_t, \tilde j_t$ with
\begin{align*}
	\supp(\tilde \rho_t) \cup  
	\supp(\tilde j_t) \subset K
			\tand
	 0<l\leq \tilde \rho_t(x) \leq \tilde M <M \, ,
\end{align*}
for a compact set $K \Subset \setR$,  $l>0$ (both depending on $\bar M$ and $\eta$), and $\tilde M<M$ (only depending on $\bar M$), as well as satisfying the  bounds for all $s \in [0,1]$
\begin{align}	\label{eq:bounds_tilde}
	\begin{gathered}
		\int_0^1 \Phi_{2,m}(\tilde \mu_t, \tilde \nu_t) \de t \leq \int_0^1 \Phi_{2,m}(\mu_t, \nu_t) \de t + \eta ,\\ 
	\max
	\left\{ 
		\Big( \text{diam}(\supp (\tilde \mu_s))^{2-q} \vee 1 \Big)
			\bW_q(\mu_s, \tilde \mu_s)^q 
		\, ,  
			\int_{\supp (\tilde \mu_s)^c} (|x|^2+1) \de \mu_s(x)
	\right\}			
\leq \eta \, .
	\end{gathered}
\end{align}
Consequently, we apply Proposition \ref{prop:upper_regular} to $(\tilde \bfmu, \tilde \bfnu)$ and find 
 curves $\tilde x^N \in C^\infty([0,1];\cK_N)$ such that
	\begin{gather}	\label{eq:proof_JKO_controls}
		\begin{gathered}
		\tilde \mu_0^N := \E^N(\tilde w^N) \to  \tilde \mu_0
			\, , \quad
		\tilde \mu_1^N := \E^N(\tilde y^N) \to \tilde \mu_1 
			\quad 
		\text{in }	\bW_\infty \, , \\
		\limsup_{N \to \infty} \int_0^1 \Phi_{2,m}^N(\tilde x^N(t), \dot{\tilde x}^N(t)) \de t \leq \int_0^1 \Phi_{2,m}(\tilde \mu_t, \tilde \nu_t) \de t \, ,
		\end{gathered}
	\end{gather}
together with the regularity properties 
for $i=0, \dots, N$, $t \in [0,1]$, $N \in \setN$
	\begin{align}	\label{eq:proof_regularity_xtilde}
		\begin{gathered}
			\tilde x_i^N(t) \in K
				\tand 
			l \leq  R_i^N(\tilde x^N(t)) \leq \tilde M \, .
		\end{gathered}
	\end{align}

We split the analysis of the problem into three different areas: one is inside $\tilde S:= \supp(\tilde \mu_0) = \supp(\tilde w^N)$ (by Proposition \ref{prop:upper_regular}), a second one far away from the compact set $K$ (and thus from $\tilde S$), and the remaining area in between. Precisely, we define
\begin{align*}
	i_1 
		&:= \inf 
	\left\{ 
		i \suchthat w_i^N > \inf(B_1(K))
	\right\} 
		\, , \quad 
	&&i_2:= \sup
	\left\{
		i \suchthat w_i^N < \sup(B_1(K))
	\right\} \, , \\
	\tilde i_1 
		&:= \inf
	\left\{ 
		i \suchthat w_i^N \geq \inf(\tilde S)
	\right\} 
		\, , \quad 
	&&\tilde i_2:= \sup
		\left\{
			i \suchthat w_i^N \leq \sup(\tilde S)
		\right\} \, ,
\end{align*}
where $B_1(K) = \{ x \in \R \ : \ d(x,K) < 1 \}$. See Figure~\ref{fig:limsup}.
\medskip
\noindent
\textit{Step 1: Definition of $y^N$.}\
We define for every $N\in\setN$ the particles $y^N \in \cK_N$ as follows: for $i \in \{ i_1 , \dots , i_2 \}$, we choose the particles associated with $\tilde y^N$, otherwise we select the particles associated with $w^N$.
Precisely, we define
\begin{align*}
	y_i^N := 
	\begin{cases}
		\tilde y_i^N \, , 	
			&\text{if } i_1 \leq i \leq i_2 \, , \\
		w_i^N \, , 
			&\text{otherwise} \, ,
	\end{cases}
\end{align*}
see Figure \ref{fig:limsup}. We then set $\mu_1^N:=\E^N(y^N)$ and claim it to be the sought recovery sequence.
	\begin{figure}[h!]
		\includegraphics[scale=2.1]{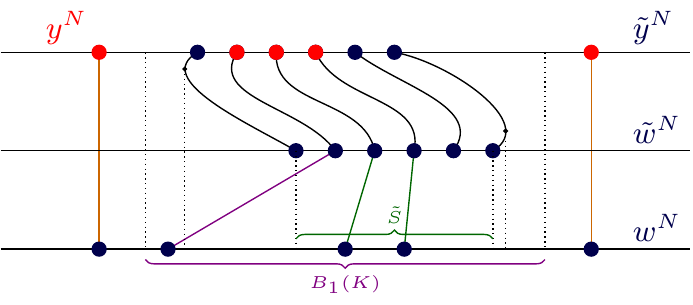}
		\caption{The definition of the recovery sequence $y^N$. }
			\label{fig:limsup}
	\end{figure}

Our first goal is to estimate $d_{2,m}^N(w^N, y^N) = \bW_{2,m}^N(\mu_0^N, \mu_1^N)$ and to show \eqref{i:lem:limsup_JKO_3}. 
 To do so, we shall introduce an intermediate step $z^N$, defined as
 \begin{align*}
	z_i^N := 
		\begin{cases}
			\tilde w_i^N \, , 	
			&\text{if } i_1 \leq i \leq i_2 \, , \\
			w_i^N \, , 
			&\text{otherwise} \, , \\
		\end{cases}
\;	= 
		\begin{cases}
			\tilde w_i^N \, , 	
			&\text{if }  i_1 \leq i \leq  i_2 \, , \\
			y_i^N \, , 
			&\text{otherwise} \, .
		\end{cases}
\end{align*}
For simplicity, we denote by $I_N$ the set of indexes corresponding to $B_1(K) \setminus \tilde S$, i.e. 
\begin{align*}
	I_N^- := \{ i \suchthat i_1 \leq i < \tilde i_1 \}
		\, , \quad 
	I_N^+:= \{ i \suchthat \tilde i_2 < i \leq i_2 \}
		\, , \quad
	I_N:= I_N^- \cup I_N^+ \, . 
\end{align*}

\noindent
\textit{Step 3: estimate of $d_{2,m}^N(w^N, z^N)$.} \ 
The cost of transporting $w^N$ onto $z^N$ consists in three parts. A trivial one, outside $B_1(K)$, where the particles do not move. The second one is the cost of connecting particles which live inside $\tilde S$, where the $\bW_q$ and the $\bW_2$ are equivalent (up to a factor depending on $\diam(\tilde S)$). The last one is the cost of transporting the particles $w^N$ from $B_1(K) \setminus \tilde S$ to the corresponding particles $z^N$ (which live in $\tilde S$ and coincide with $\tilde w^N$). See Figure~\ref{fig:limsup}. We estimate that in terms of the tails of the second moments of $\mu_0$, see \eqref{eq:step3_3}.

Note that $w_{i_1-1}^N, w_{i_2+1}^N \in \R \setminus B_1(K)$, it is immediate to see that $R_{i_1-1}^N(z^N) \vee R_{i_2}^N(z^N) \leq 1/N$ for $N$ big enough.
By \eqref{eq:assumpt_Case1}, \eqref{eq:proof_regularity_xtilde}, we have 
\begin{align*}
	\sup_{N \in \setN} \sup_{i\leq N} R_i^N(z^N) \leq M'\, , 
	\quad 
		\text{where} \quad 
		M' = \max \{ \tilde M, \bar M \} < M 
\end{align*}
does not depend on $\eta$. 
In particular, we obtain that
\begin{align}
	\begin{aligned}
		\label{eq:step3_1}
	d_{2,m}^N(w^N, z^N)^2 \leq c_\theta 
		\left\|
			w^N - z^N
		\right\|_{\ell_{2,N}}^2
	= \frac{c_\theta}{N+1} \sum_{i=i_1}^{i_2} \left|
	w_i^N - z_i^N
	\right|^2
	\, ,
	\end{aligned}
\end{align}
where $c_\theta := \theta(M')^{-1} < \infty$. 
By construction, on one side we have that (green lines in Figure~\ref{fig:limsup})
\begin{align}
	\begin{aligned}
		\label{eq:step3_2}
	\frac1{N+1} \sum_{i=\tilde i_1}^{\tilde i_2}
		\left|
			w_i^N - z_i^N
		\right|^2
	&\leq
		\frac1{N+1} \diam(\tilde S)^{2-q} 
			\sum_{i=\tilde i_1}^{\tilde i_2}
		\left|
			w_i^N - \tilde w_i^N
		\right|^q	 \\
	&\leq 
		\diam(\tilde S)^{2-q} \bW_q(\mu_0^N, \tilde \mu_0^N)^q \, .
		\end{aligned}
\end{align}

We are left to estimate the remaining terms of the sum in \eqref{eq:step3_1}, i.e. corresponding to $i \in I_N$. By \eqref{eq:bounds_tilde}, we have that $\mu_0((\tilde S )^c) \leq \eta$, whereas by assumption we know that $\mu_0^N \to \mu_0$ vaguely. Therefore, using that $\mu_0$ is absolutely continuous with respect to $\Leb^1$ (in particular, $\mu_0(\partial \tilde S)=0$), we get $\mu_0^N(\tilde S^c) \leq 2\eta$ for $N$ large enough. As a consequence, we infer that, for $\eta < \frac14$, 
\begin{align}
		\label{eq:bound_median}
	\forall  i \in I_N^-, \quad \tilde w_i^N \leq x_{\tilde \mu_0^N}, \quad 
	\forall  i \in I_N^+, \quad \tilde w_i^N \geq x_{\tilde \mu_0^N} \, ,
\end{align}
where $x_{\tilde \mu_0^N} = X_{\tilde \mu_0^N}(\frac12)$ is the median of $\tilde \mu_0^N$.
Note that $w_i^N \in B_1(K) \setminus \tilde S$ if and only if $i \in I_N$ (corresponding to the violet lines in Figure~\ref{fig:limsup}). Together with \eqref{eq:bound_median}, this implies 
\begin{align}
	\nonumber
	\frac1{N+1} \sum_{i \in I_N}
		\left|
			w_i^N - z_i^N
		\right|^2
	&= 
	\frac1{N+1} \sum_{i \in I_N}
		\left|
			w_i^N - \tilde w_i^N
		\right|^2 
\leq 
		\frac1N \sum_{i \in I_N}
		\left|
			w_i^N - x_{\tilde \mu_0^N} 
		\right|^2 
\\ \label{eq:step3_3}
&=
		\int_{B_1(K) \setminus \tilde S}
			\left|
				x - x_{\tilde \mu_0^N}
			\right|^2
				\de \mu_0^N(x)
				\, .
\end{align}
Thanks to the fact that $\| X_{\tilde \mu_0^N} - X_{\tilde \mu_0} \|_{L^\infty(0,1)} \to 0$ by \eqref{eq:proof_JKO_controls} and $\mu_0^N \to \mu_0$ in $\bW_q$, by \eqref{eq:step3_3} we obtain that
\begin{align}
	\begin{aligned}
		\label{eq:step3_4}
	\limsup_{N \to \infty}
		\frac1{N+1} \sum_{i \in I_N}
		\left|
		w_i^N - z_i^N
		\right|^2
	&\leq 
		\int_{B_1(K) \setminus \tilde S}
		\left|
		x - x_{\tilde \mu_0}
		\right|^2
		\de \mu_0(x)
		\\
	&\leq 
		2
		\int_{\tilde S^c}
			| x |^2
		\de \mu_0(x)
		+ 2 \eta
		\left| x_{\tilde \mu_0} \right|^2 
	\end{aligned}
\end{align}
where we used $\mu_0(\partial \tilde S)=0$ and $\mu_0(\tilde S^c) \leq \eta$. 
Altogether, using \eqref{eq:step3_1}, \eqref{eq:step3_2}, \eqref{eq:step3_4}, and the fact that $\mu_0^N \to \mu_0$ and $\tilde \mu_0^N \to \tilde \mu_0$ in $\bW_q$, taking the limsup in $N \to \infty$ we obtain 
\begin{align}	\nonumber
	\limsup_{N \to \infty} d_{2,m}^N(w^N, z^N)^2
		&\leq 
			2 c_\theta 
			\left(
				\diam(\tilde S)^{2-q} \bW_q(\mu_0, \tilde \mu_0)^q
					+
				\int_{\tilde S^c}
					|
						x 
					|^2
				\de \mu_0(x)
			+ 
				 \eta
				\left| x_{\tilde \mu_0} \right|^2
			\right) 
		\\
		&\leq \label{eq:bound_step2}
			2c_\theta \eta \left( 2  + \left| x_{\tilde \mu_0} \right|^2 \right) 
			 \, ,
\end{align}
where at last we used \eqref{eq:bounds_tilde}.

\medskip
\noindent
\textit{Step 4: estimate of $d_{2,m}^N(z^N, y^N)$.} \
We bound from above the discrete distance from $z^N$ to $y^N$ using the competitor
\begin{align*}
	x_i^N(t) := 
	\begin{cases}
		\tilde x_i^N(t) \, , 
		&\text{if }  i_1 \leq i \leq  i_2  \, , \\
		w_i^N \, ,
		&\text{otherwise} \, ,
	\end{cases}
\end{align*}
which is admissible by construction, i.e. $x^N(0)=z^N$, $x^N(1) = y^N$, and $x^N \in C^\infty([0,1] ; \cK_N)$, thanks to the fact that $( \tilde x_{i_1}^N(\cdot) - w_{i_1-1}^N) \wedge (w_{i_2+1}^N -\tilde x_{i_2}^N(\cdot)) \geq 1$.
Since $\phi \geq 0$ and $R_{i_2}^N(x^N(\cdot)) \leq R_{i_2}^N(\tilde x^N(\cdot))$, we obtain the estimate
\begin{align*}
	d_{2,m}^N(z^N,y^N)^2 \leq \int_0^1 \Phi_{2,m}^N(\tilde x^N(t), \dot{\tilde x}^N(t)) \de t \, .
\end{align*}
Using first \eqref{eq:proof_JKO_controls} and then \eqref{eq:bounds_tilde}, we conclude
\begin{align}	\label{eq:bound_step3}
	\limsup_{N \to \infty} d_{2,m}^N(z^N,y^N)^2 
		\leq
			\int_0^1 \Phi_{2,m}(\tilde \mu_t, \tilde \nu_t) \de t 
		\leq  \btwoW(\mu_0,\mu_1)^2 + \eta  \, .
\end{align}

\medskip
\noindent
\textit{Step 5: proof of \eqref{i:lem:limsup_JKO_3}.} \
Putting together \eqref{eq:bound_step2} and \eqref{eq:bound_step3} we conclude  for every $\eta>0$
\begin{align*}
	\limsup_{N \to \infty} \bW_{2,m}^N(\mu_0^N, \mu_1^N) 
		&\leq 
	\limsup_{N \to \infty}
		d_{2,m}^N(w^N,z^N)
			+
	 	d_{2,m}^N(z^N,y^N)
	 \\
	 &\leq 
	 	\eta^{\frac12}(2c_\theta)^{\frac12} \left( 2  + \left| x_{\tilde \mu_0} \right|^2 \right)^{\frac12}
	 +
		\Big( 
			\btwoW(\mu_0, \mu_1)^2
				+
			\eta 
		\Big)^{\frac12}	
			    \, .
\end{align*}
Finally, by \eqref{eq:bounds_tilde} for $s=0$ we know that $\tilde \mu_0 \to \mu_0$ in $\bW_q$ as $\eta \to 0$. As a consequence of Lemma~\ref{lemma:stability_Ieps}\eqref{i:lemma:stability_Ieps:diam_tight}, the sought bound \eqref{i:lem:limsup_JKO_3} then follows from choosing $\eta$ sufficiently small. 

\medskip 
\noindent
\textit{Step 6: conclusion}. \
We note that, by the assumption \eqref{eq:unif_bounds_2ndmoms} on $\mu_0^N$ and \eqref{eq:proof_JKO_controls},
\begin{align*}
	\sup_{N \in \N} 
		\bE_{\mu_1^N} |x|^2 
	\leq
	\sup_{N \in \N}
	\big(
		\bE_{\mu_0^N} |x|^2 
			+
		\bE_{\tilde \mu_1^N} |x|^2 
	\big)
		< \infty \, ,
\end{align*}
therefore there exists $\bar \mu_1 \in \cP_2(\R)$ such that, up to a (non-relabeled) subsequence, we get that $\mu_1^N \to \bar  \mu_1$ in $\bW_q$.
We are left with the proof of \eqref{i:lem:limsup_JKO_2}. By \eqref{eq:bounds_tilde}, we have that
\begin{align}
		\label{eq:proof_final_1}
	\bW_q(\mu_1, \bar \mu_1) \leq 
		\bW_q(\mu_1,  \tilde \mu_1)
	+
		\bW_q(\tilde \mu_1, \bar \mu_1)
	\leq 
		\eta^{\frac1q} + 
			\bW_q(\tilde \mu_1, \bar \mu_1) \, .
\end{align}
On the other hand, by $\tilde \mu_1^N \to \tilde \mu_1$ in $\bW_\infty$ and arguing as in \eqref{eq:bound_median}--\eqref{eq:step3_4}, we find that
\begin{align*}
	\bW_q&(\tilde \mu_1, \bar \mu_1)
		=
	\lim_{N \to \infty}
		\bW_q(\tilde \mu_1^N, \mu_1^N)^q
		=
	\lim_{N \to \infty} 
		\frac1{N+1} \sum_{i \notin \{i_1, \dots , i_2\}}
		\left|
			w_i^N - \tilde y_i^N
		\right|^q
\\
		&\leq
	\int_{B_1(K)^c} 
				\left|
		x - x_{\tilde \mu_1}
		\right|^q
		\de \mu_0(x)		
		\leq
	2
	\int_{\tilde S^c}
		| x |^q
			\de \mu_0(x)
	+ 2 \eta
		\left| x_{\tilde \mu_1} \right|^q 
		\leq 
	2 \eta \left( 1 + \left| x_{\tilde \mu_1} \right|^q \right) 
	\,  , 
\end{align*}
where at last we used \eqref{eq:bounds_tilde} once more. Together with \eqref{eq:proof_final_1}, we deduce that
\begin{align*}
	\bW_q(\mu_1, \bar \mu_1) 
		\leq
	 \eta^{\frac1q} + 
	(2 \eta)^{\frac1q} \left( 1 + \left| x_{\tilde \mu_1} \right|^q \right)^{\frac1q} \, .  
\end{align*}
Finally, by $\tilde \mu_1 \to \mu_1$ as $\eta \to 0$ and Lemma~\ref{lemma:stability_Ieps}\eqref{i:lemma:stability_Ieps:diam_tight}, for every given $\eps >0$ we obtain the sought bound \eqref{i:lem:limsup_JKO_2} choosing  $\eta$  sufficiently small.

\medskip
\noindent 
\underline{Case 2}: \
We now turn to the general case, meaning we do not assume that \eqref{eq:assumpt_Case1} is satisfied.
	For $\lambda>0$, we consider the regularised measures $(\mu_t^\lambda , \nu_t^\lambda) \in \CE$ given by
\begin{align*}
	\mu_t^\lambda := \big( (1+\lambda)\text{id} \big)_{\#} \mu_t
	\, , \quad 
	\nu_t^\lambda := (1+\lambda) 
	\big( (1+\lambda)\text{id} \big)_{\#} \nu_t \, .
\end{align*}
Note that by construction, the corresponding densities $\rho_t^\lambda$, $j_t^\lambda$ are given by
\begin{align}
		\label{eq:density_lambda}
	\rho_t^\lambda(x) = \frac{1}{1+\lambda} \rho_t 
	\Big(
	\frac{x}{1+\lambda}
	\Big) 
	\, , \quad 
	j_t^\lambda(x) = j_t
	\Big(
	\frac{x}{1+\lambda}
	\Big) \, ,	
	\quad 
	\forall x \in \setR \, , t \in [0,1] \, .
\end{align}
Furthermore, we have that $\mu_0^{N,\lambda} := \big( (1+\lambda)\text{id} \big)_{\#} \mu_0^N= \E^N(w^{N,\lambda})$, where $w^{N,\lambda} = (1+\lambda) w^N$. In particular, we note that  
\begin{align}
	\label{eq:density_bound_w_N_lambda}
	\max 
	\left\{
		\sup_{N \in \setN} \sup_{i\leq N} R_i^N(w^{N,\lambda})
	, \, 
		\sup_t \| \rho_t^\lambda \|_\infty
	\right\}
		 \leq (1+\lambda)^{-1} M < M \, . 
\end{align}
We can then apply the first part of the proof (Case 1) to $\mu_0^{N,\lambda}$, $\mu_0^\lambda$, and $\mu_1^\lambda$: for any fixed $\eps>0$, we can find $\mu_1^{N,\lambda} \in \cE^N$, $\bar \mu_1^\lambda \in \cP_2(\R)$ such that $\mu_1^{N,\lambda} \to \bar \mu_1^\lambda$ in $\bW_q$ and 
\begin{align}
\label{eq:energy_estimate_case2}
	\bW_q( \mu_1^\lambda, \bar \mu_1^\lambda) \leq \frac{\eps}2
		\tand 
	\limsup_{N \to \infty} 			
		\bW_{2,m}^N(\mu_0^{N,\lambda}, \mu_1^{N,\lambda}) 
	\leq 
		\btwoW(\mu_0^\lambda, \mu_1^\lambda)	
			+ \frac{\eps}2 \, .
\end{align}
Using similar arguments as in \cite[Theorem 4.1]{lisiniMarigonda2010}, it is not hard to see that
\begin{align}	\label{eq:bound_step1}
	\sup_{N \in \setN}	d_{2,m}^N(w^N, w^{N,\lambda})^2 \leq C \, \lambda \, \sup_N \|w^N \|_{\ell_{2,N}}^2 =: \bar C^2 \lambda   \, , 
\end{align}
where $C\in \R_+$ only depends on $\phi$ and $\bar C< \infty$ since  $\seq{\mu_0^N}_N$ satisfies \eqref{eq:unif_bounds_2ndmoms}.
Moreover, by \eqref{eq:density_lambda}, we also have that
\begin{align}	\label{eq:energy_bound_lambda}
	\int_0^1 \Phi_{2,m}(\mu_t^\lambda, \nu_t^\lambda) \de t \leq (1+\lambda) \int_0^1 \Phi_{2,m}((1+\lambda)^{-1} \mu_t, \nu_t) \de t 	
	\leq
	(1+\lambda) \btwoW(\mu_0,\mu_1)^2 \, ,
\end{align}
where at last we used that $\theta$ is non-increasing (see Remark~\ref{rem:monotonicity_theta}). Putting \eqref{eq:energy_estimate_case2}, \eqref{eq:bound_step1}, and \eqref{eq:energy_bound_lambda} together, we obtain that
\begin{align*}
	\limsup_{N \to \infty} 			
		\bW_{2,m}^N(\mu_0, \mu_1^{N,\lambda}) 
	&\leq 
	\limsup_{N \to \infty}
		\bW_{2,m}^N(\mu_0, \mu_0^{N,\lambda}) 
			+ 
		\bW_{2,m}^N(\mu_0^{N, \lambda}, \mu_1^{N,\lambda})
\\
	&\leq 
		\bar C \lambda^{\frac12} + \sqrt{(1+\lambda)} \btwoW(\mu_0,\mu_1) + \frac{\eps}2  \, .
\end{align*}
This shows that \eqref{i:lem:limsup_JKO_3} is satisfied with $\mu_1^N:= \mu_1^{N,\lambda}$ if $\lambda>0$ is sufficiently small. Finally, by \eqref{eq:energy_estimate_case2} and the scaling properties of $\bW_q$, we have that
\begin{align*}
	\bW_q( \mu_1, \bar \mu_1^\lambda)
		\leq
	\bW_q( \mu_1 ,  \mu_1^\lambda) 
		+ 
	\bW_q( \mu_1^\lambda, \bar \mu_1^\lambda)
		\leq 
	\lambda \big( \bE_{\mu_1}|x|^q \big)^{\frac1q} + \frac{\eps}2  \, .
\end{align*} 
This shows that \eqref{i:lem:limsup_JKO_2} is satisfied with $\bar \mu_1:= \bar \mu_1^\lambda \in \cP_2(\R)$ if $\lambda>0$ is sufficiently small. Given that $\mu_1^N \to \bar \mu_1$ in $\bW_q$ by construction, we also deduce  \eqref{i:lem:limsup_JKO_1}, which ends the proof.
%
%
\end{proof}

\subsection{Proof of the convergence of the JKO schemes (Theorem \ref{thm:JKO_q})}
We are finally ready to prove Theorem \ref{thm:JKO_q}.

\begin{proof}[Proof of Theorem \ref{thm:JKO_q}]
We start showing \eqref{i:thm:JKO_q_1}.  
	With no loss of generality, we assume $\tau =\frac12$ and $F_N \equiv 0$, the general case following trivially by the assumption in \eqref{eq:cont_F_N_F}. Thus, $\cJ_{\tau,N}(\cdot) = \bW_{2,m}^N(\mu_N^{(0)}, \cdot)^2$ and  $\cJ_\tau(\cdot) = \btwoW(\mu^{(0)}, \cdot)^2$. In this case, the $\Gamma$-$\liminf$ inequality easily follows by Theorem~\ref{thm:Gamma-conv}(LI), hence we shall prove the $\Gamma$-$\limsup$.
	
	Let $\mu \in \cP_2(\R)$ such that $\btwoW(\mu^{(0)},\mu)<\infty$. For every $\eps>0$, we apply Proposition~\ref{prop:limsup_JKO} and find a sequence $\seq{  \mu^{N,\eps} \in \cE^N}_N$ and $\bar  \mu^\eps \in \cP_2(\R)$ such that $\mu^{N,\eps} \to \bar \mu^\eps$ in $\bW_q$, with $\bW_q(\mu, \bar \mu^\eps) \leq \eps$, and $
	\limsup_N
	\bW_{2,m}^N
	\big(
	\mu_N^{(0)},  \mu^{N,\eps} 
	\big)
	\leq
	\btwoW(\mu^{(0)}, \mu) + \eps$. Hence, we can choose a diagonal sequence $\mu^N:= \mu^{N,\eps_N}$ such that $\mu^N \to \mu$ in $\bW_q$ and 
	\begin{align*}
			\limsup_{N \to \infty}
		\bW_{2,m}^N
		\big(
		\mu_N^{(0)}, \mu^N 
		\big)
		\leq
		\btwoW(\mu^{(0)}, \mu) \, ,  
	\end{align*}
which concludes the proof of \eqref{i:thm:JKO_q_1}.

Finally, the proof of \eqref{i:thm:JKO_q_2} follows from the $\Gamma$-convergence shown in \eqref{i:thm:JKO_q_1} and the equicoercivity property proved in Proposition~\ref{prop:equicoercivity}, see e.g.  \cite[Theorem~1.21]{braides2002}.

\subsection*{Acknowledgments}
This work started when L.P. and S.dM visited E.R. at the University of l'Aquila and continued first at the Erwin Schrödinger Institute in Vienna, then at the École Polytechnique Fédérale de Lausanne and also in the Simons institute in Berkeley, CA.
The authors thank the hospitality of these three institutions and, especially, Marco di Francesco for suggesting the topic in first place and for fruitful discussions and literature suggestions.
S.dM acknowledges support by GNAMPA (INdAM). L.P.  acknowledges support by the Austrian Science Fund (FWF), grants No W1245 and
No F65, and by the Hausdorff Center for Mathematics in Bonn. E.R. acknowledges support by GNAMPA (INdAM).
\end{proof}
\bibliographystyle{alpha}
\bibliography{biblioprop} 

\tocless\section{Appendix A. Properties of the action densities}
The following lemma concerns continuity properties of $\phi$.
\begin{lemma}[]	\label{lemma:rescaling}
	Fix $\phi_{p,m}$ as in Defintion~\ref{def:action_density_functions} and define $g(r):= \phi_{p,m}(r,1)$. Pick $\bar M < M$ and for every $\lambda \in (1,2)$, consider the set
	\begin{align*}
		I_{\bar M}(\lambda):= 
		\left\{
		r \in (0,M) \suchthat \lambda r \leq \bar M 
		\right\} \, , 
		\quad
		I_{\bar M}(\lambda) \subset (0,M) \, .
	\end{align*}
	Then we have:
	\begin{align}
		\label{eq:estimate-g-dilatation}
		g(\lambda r) \leq \big( 1 + C_{\bar M} (\lambda -1) \big) g(r) \, , \quad \forall r \in I_{\bar M}(\lambda) \, , 
	\end{align}
	where $C_{\bar M} < \infty$ is a constant depending on $g$ and $\bar M$.
\end{lemma}

\begin{proof}
	Let $r \in I_{\bar M}(\lambda)$. Note that $g$ is strictly positive convex function with domain $(0,M)$ such that $g\to +\infty$ at the boundary of this interval. Hence there exists $r_{\min} \in (0,M)$ a minimum of $g$. 
	We split the proof of \eqref{eq:estimate-g-dilatation} into two cases.
	
	\medskip
	\noindent
	\textit{Case 1}: \
	$\lambda r \leq r_{\min}$. In this case, having that $\lambda >1$ and $g$ is non increasing (by convexity) in $(0,r_{\min})$, we conclude that $g(\lambda r) \leq g(r)$.
	
	\medskip
	\noindent
	\textit{Case 2}: \
	$\lambda r > r_{\min}$. Using that $\lambda < 2$, we obtain that $r > r_{\min}/2$ and thus $r,\lambda r \in (r_{\min}/2, \bar M)$. Over this interval, $g$ is a Lipschitz map of Lipschitz constant $\bar L=\bar L(r_{\min}, \bar M)$. It follows
	\begin{align*}
		g(\lambda r) 
		\leq g(r) + \bar L (\lambda - 1)r		
		\leq  
		\bigg(
		1 + \frac{\bar L \bar M}{g(r_{\min})} 
		(\lambda -1)
		\bigg)
		g(r) \, , 		
	\end{align*}
	which concludes the proof.
\end{proof}

\end{document}